\documentclass[a4paper]{amsart}
\usepackage[utf8]{inputenc}
\usepackage{eucal}
\usepackage[english]{babel}
\usepackage{amsmath}
\usepackage{amssymb}
\usepackage{amsopn}                 
\usepackage{amsbsy}
\usepackage{mathtools}              
\usepackage{mathrsfs}
\usepackage{relsize}
\usepackage{enumerate}
\usepackage{enumitem}
\usepackage{comment}
\usepackage{float}
\usepackage[abbreviations]{foreign} 
\usepackage{xy}
\xyoption{arrow}
\xyoption{matrix}
\SilentMatrices
\xyoption{tips}
\SelectTips{cm}{10}
\xyoption{2cell}
\UseAllTwocells
\usepackage{tikz}
\usetikzlibrary{cd}
\usetikzlibrary{calc}
\usetikzlibrary{math}
\usetikzlibrary{arrows}
\usetikzlibrary{fit}
\usetikzlibrary{positioning}

\usetikzlibrary{decorations.pathmorphing, arrows.meta}

\tikzset{Rightarrow/.style={double equal sign distance,>={Implies},->},
triple/.style={-,preaction={draw,Rightarrow}},
quadruple/.style={preaction={draw,Rightarrow,shorten >=0pt},shorten >=1pt,-,double,double
distance=0.2pt}}
\usepackage{xcolor}
\definecolor{darkblue}{rgb}{0,0,0.3}
\usepackage[colorlinks=true, citecolor=darkblue, filecolor=darkblue, linkcolor=darkblue, urlcolor=darkblue]{hyperref}
\newtheorem{thm}{Theorem}[section]
\newtheorem{cor}[thm]{Corollary}

\newtheorem{lemma}[thm]{Lemma}
\newtheorem{prop}[thm]{Proposition}

\theoremstyle{definition}
\newtheorem{define}[thm]{Definition}
\newtheorem{notate}[thm]{Notation}

\theoremstyle{remark}
\newtheorem{rem}[thm]{Remark}
\newtheorem{example}[thm]{Example}

{\parskip=0pt\par\nopagebreak\centering}
{\par\noindent\ignorespacesafterend}        
\newcommand\nbd\nobreakdash
\newcommand{\ndef}{\emph}

\def\lrar{\longrightarrow}
\newcommand{\bS}{\mathbf{S}}        
\newcommand{\op}{\text{op}}

\newcommand{\Cat}{{\mathcal{C}\mspace{-2.mu}\mathit{at}}}
\newcommand{\nCat}[1]{{#1}\hbox{\protect\nbd-}\kern1pt\Cat}	    
\newcommand{\s}{\mathcal{S}\mspace{-2.mu}\text{et}_{\Delta}}
\newcommand{\Ss}{\mathcal{S}\mspace{-2.mu}\text{et}_{\Delta}^{\,\mathrm{sc}}}

\newcommand{\Nsc}{\mathrm{N}^{\mathrm{sc}}}			
\newcommand{\ho}{\mathrm{ho}}
\def\Beta{\mathfrak{P}}                             
\newcommand{\ogr}{\otimes}
\newcommand{\hgr}{\hat\ogr}
\newcommand{\Fun}{\mathrm{Fun}}

\DeclareMathOperator{\Hom}{Hom}

\DeclareMathOperator{\Map}{Map}

\DeclareMathOperator{\Id}{Id}
\DeclareMathOperator{\Ob}{Ob}
\DeclareMathOperator{\Ho}{Ho}

\DeclareMathOperator{\Set}{Set}

\DeclareMathOperator{\sca}{sc}
\newcommand{\gr}{\mathrm{gr}}               
\newcommand{\opgr}{\mathrm{opgr}}
\newcommand{\LMap}{\Fun^{\opgr}}
\newcommand{\RMap}{\Fun^{\gr}}
\DeclareMathOperator{\diag}{diag}

\DeclareMathOperator{\thi}{th}

\DeclareMathOperator{\Psh}{PSh}
\def\alp{{\alpha}}
\def\bet{{\beta}}

\def\sig{{\sigma}}
\def\vphi{{\varphi}}

\def\Del{{\Delta}}

\def\Lam{{\Lambda}}

\def\vphi{{\varphi}}

\def\hrar{\hookrightarrow}

\def\x{\stackrel}

\def\ovl{\overline}

\def\wtl{\widetilde}

\newcommand{\tr}[2]{\mathchoice
	{#1\raise -1.8pt\vbox{\hbox{$\kern -.8pt/\mathsmaller{#2} $}}}
	{#1\raise -1.8pt\vbox{\hbox{$\kern -.8pt/#2$}}\kern .8pt}
	{#1\raise -1.8pt\vbox{\hbox{$\scriptstyle\kern -.8pt /#2$}}}
	{#1\raise -1.8pt\vbox{\hbox{$\scriptscriptstyle\kern -.8pt /#2$}}}}

\newcommand{\trbis}[2]{\mathchoice
	{#1\raise -1.8pt\vbox{\hbox{$\kern -.8pt\mathsmaller{/#2} $}}}
	{#1\raise -1.8pt\vbox{\hbox{$\kern -.8pt\mathsmaller{/#2}$}}\kern .8pt}
	{#1\raise -1.8pt\vbox{\hbox{$\scriptstyle\kern -.8pt /#2$}}}
	{#1\raise -1.8pt\vbox{\hbox{$\scriptscriptstyle\kern -.8pt /#2$}}}}
\newcommand{\overslice}[2]{\mathchoice
	{#1\raise -1.8pt\vbox{\hbox{$\kern -.8pt\mathsmaller{#2/} $}}}
	{#1\raise -1.8pt\vbox{\hbox{$\kern -.8pt\mathsmaller{#2/}$}}\kern .8pt}
	{#1\raise -1.8pt\vbox{\hbox{$\scriptstyle\kern -.8pt #2/$}}}
	{#1\raise -1.8pt\vbox{\hbox{$\scriptscriptstyle\kern -.8pt #2/$}}}}

	\makeatletter

\def\labelstylecode#1{%
	\pgfkeys@split@path%
	\edef\label@key{/triangle/label/\pgfkeyscurrentname}%
	\edef\style@key{\pgfkeyscurrentkey/.@val}%
	\def\temp@a{#1}%
	\def\temp@b{\pgfkeysnovalue}%
	\ifx\temp@a\temp@b
	\pgfkeysgetvalue{\label@key}\temp@a
	\ifx\temp@a\temp@b\else
	\pgfkeysalso{commutative diagrams/.cd, \style@key}%
	\fi
	\else
	\pgfkeys{\style@key/.code = \pgfkeysalso{#1}}%
	\fi}
\def\arrowstylecode#1{%
	\edef\style@key{\pgfkeyscurrentkey/.@val}%
	\def\temp@a{#1}%
	\def\temp@b{\pgfkeysnovalue}%
	\ifx\temp@a\temp@b
	\pgfkeysalso{commutative diagrams/.cd, \style@key}%
	\else
	\pgfkeys{\style@key/.code = \pgfkeysalso{#1}}%
	\fi}

\pgfkeys{
	/triangle/label/.cd,
	0/.initial = {$\bullet$}, 1/.initial = {$\bullet$}, 2/.initial = {$\bullet$},
	01/.initial, 12/.initial, 02/.initial,
	012/.initial,
	/triangle/labelstyle/.cd,
	01/.@val/.initial=\pgfkeysalso{swap}, 12/.@val/.initial=\pgfkeysalso{swap},
	02/.@val/.initial,
	012/.@val/.initial=\pgfkeysalso{near start},
	01/.code=\labelstylecode{#1}, 02/.code=\labelstylecode{#1},
	12/.code=\labelstylecode{#1}, 012/.code=\labelstylecode{#1},
	/triangle/arrowstyle/.cd,
	01/.@val/.initial, 12/.@val/.initial,
	02/.@val/.initial,
	012/.@val/.initial=\pgfkeysalso{Rightarrow},
	01/.code=\arrowstylecode{#1}, 02/.code=\arrowstylecode{#1},
	12/.code=\arrowstylecode{#1}, 012/.code=\arrowstylecode{#1},
}

\def\tr@abc{%
	\draw [/triangle/arrowstyle/012] (90:0.20) --
	node [/triangle/labelstyle/012] {
		\pgfkeysvalueof{/triangle/label/012}} (270:0.10);
}

\def\tr@#1#2{
	\begin{scope}[shift=#2, commutative diagrams/every diagram]
		
		\node (n{#1}0) at (150:1) {
			\pgfkeysvalueof{/triangle/label/0}};
		\node (n{#1}1) at (270:0.6) {
			\pgfkeysvalueof{/triangle/label/1}};
		\node (n{#1}2) at (30:1) {
			\pgfkeysvalueof{/triangle/label/2}};			
		
		\node (s#1) at (0,0) [circle, inner sep = 0pt,
		fit = (n{#1}0.center)(n{#1}1.center)(n{#1}2.center)] {};
		
		\begin{scope}[commutative diagrams/.cd, every arrow, every label]
			\ifcase #1
			\def\list{0/1, 1/2, 0/2}\or
			\def\list{0/1, 1/2, 0/2}\else
			\def\list{}\fi
			
			\foreach \s / \e in \list {
				\draw [/triangle/arrowstyle/\s\e] (n{#1}\s) --
				node [/triangle/labelstyle/\s\e] {
					\pgfkeysvalueof{/triangle/label/\s\e}} (n{#1}\e);
			}
			
			\ifcase #1
			\tr@abc\or
			\tr@abc
			\else\fi
			
		\end{scope}
	\end{scope}
}

\def\triangle#1{
	\pgfkeys{#1}
	\tr@{0}{(0:0)}
}

\makeatother

\makeatletter

\def\labelstylecode#1{%
	\pgfkeys@split@path%
	\edef\label@key{/square/label/\pgfkeyscurrentname}%
	\edef\style@key{\pgfkeyscurrentkey/.@val}%
	\def\temp@a{#1}%
	\def\temp@b{\pgfkeysnovalue}%
	\ifx\temp@a\temp@b
	\pgfkeysgetvalue{\label@key}\temp@a
	\ifx\temp@a\temp@b\else
	\pgfkeysalso{commutative diagrams/.cd, \style@key}%
	\fi
	\else
	\pgfkeys{\style@key/.code = \pgfkeysalso{#1}}%
	\fi}
\def\arrowstylecode#1{%
	\edef\style@key{\pgfkeyscurrentkey/.@val}%
	\def\temp@a{#1}%
	\def\temp@b{\pgfkeysnovalue}%
	\ifx\temp@a\temp@b
	\pgfkeysalso{commutative diagrams/.cd, \style@key}%
	\else
	\pgfkeys{\style@key/.code = \pgfkeysalso{#1}}%
	\fi}

\pgfkeys{
	/square/label/.cd,
	0/.initial = {$\bullet$}, 1/.initial = {$\bullet$}, 2/.initial = {$\bullet$},
	3/.initial = {$\bullet$},
	01/.initial, 12/.initial, 23/.initial,
	02/.initial, 03/.initial, 13/.initial,
	012/.initial, 013/.initial, 023/.initial, 123/.initial,
	0123/.initial,
	/square/labelstyle/.cd,
	01/.@val/.initial=\pgfkeysalso{swap},
	12/.@val/.initial=\pgfkeysalso{swap},
	23/.@val/.initial=\pgfkeysalso{swap},
	03/.@val/.initial,
	02/.@val/.initial=\pgfkeysalso{description},
	13/.@val/.initial=\pgfkeysalso{description},
	012/.@val/.initial=\pgfkeysalso{above left = -1pt and -1pt},
	013/.@val/.initial=\pgfkeysalso{swap, right},
	023/.@val/.initial=\pgfkeysalso{left},
	123/.@val/.initial=\pgfkeysalso{swap, above right = -1pt and -1pt},
	0123/.@val/.initial,
	01/.code=\labelstylecode{#1}, 02/.code=\labelstylecode{#1},
	03/.code=\labelstylecode{#1}, 12/.code=\labelstylecode{#1},
	13/.code=\labelstylecode{#1}, 23/.code=\labelstylecode{#1},
	012/.code=\labelstylecode{#1}, 013/.code=\labelstylecode{#1},
	023/.code=\labelstylecode{#1}, 123/.code=\labelstylecode{#1},
	0123/.code=\labelstylecode{#1},
	/square/arrowstyle/.cd,
	01/.@val/.initial, 12/.@val/.initial, 23/.@val/.initial,
	02/.@val/.initial, 03/.@val/.initial, 13/.@val/.initial,
	012/.@val/.initial=\pgfkeysalso{Rightarrow},
	013/.@val/.initial=\pgfkeysalso{Rightarrow},
	023/.@val/.initial=\pgfkeysalso{Rightarrow},
	123/.@val/.initial=\pgfkeysalso{Rightarrow},
	0123/.@val/.initial=\pgfkeysalso{Rightarrow},
	01/.code=\arrowstylecode{#1}, 02/.code=\arrowstylecode{#1},
	03/.code=\arrowstylecode{#1}, 12/.code=\arrowstylecode{#1},
	13/.code=\arrowstylecode{#1}, 23/.code=\arrowstylecode{#1},
	012/.code=\arrowstylecode{#1}, 013/.code=\arrowstylecode{#1},
	023/.code=\arrowstylecode{#1}, 123/.code=\arrowstylecode{#1},
	0123/.code=\arrowstylecode{#1}
}

\def\sq@abc{%
	\draw [/square/arrowstyle/012] (235:0.25) --
	node [/square/labelstyle/012] {
		\pgfkeysvalueof{/square/label/012}} (235:0.6);
}
\def\sq@bcd{%
	\draw [/square/arrowstyle/123] (-54:0.25) --
	node [/square/labelstyle/123] {
		\pgfkeysvalueof{/square/label/123}} (-54:0.6);
}
\def\sq@acd{%
	\draw [/square/arrowstyle/023] (55:0.55) --
	node [/square/labelstyle/023] {
		\pgfkeysvalueof{/square/label/023}} (15:0.45);
}
\def\sq@abd{%
	\draw [/square/arrowstyle/013] (125:0.55) --
	node [/square/labelstyle/013] {
		\pgfkeysvalueof{/square/label/013}} (165:0.45);
}

\def\sq@#1#2{
	\begin{scope}[shift=#2, commutative diagrams/every diagram]
		
		\foreach \i in {0,1,2,3} {
			\tikzmath{\a = 135 + (90 * \i);}
			\node (n{#1}\i) at (\a:1) {
				\pgfkeysvalueof{/square/label/\i}};
		}
		
		\node (s#1) at (0,0) [circle, inner sep = 0pt,
		fit = (n{#1}0.center)(n{#1}1.center)(n{#1}2.center)
		(n{#1}3.center)] {};
		
		\begin{scope}[commutative diagrams/.cd, every arrow, every label]
			\ifcase #1
			\def\list{0/1, 1/2, 2/3, 0/2, 0/3}\or
			\def\list{0/1, 1/2, 2/3, 1/3, 0/3}\else
			\def\list{}\fi
			
			\foreach \s / \e in \list {
				\draw [/square/arrowstyle/\s\e] (n{#1}\s) --
				node [/square/labelstyle/\s\e] {
					\pgfkeysvalueof{/square/label/\s\e}} (n{#1}\e);
			}
			
			\ifcase #1
			\sq@abc\sq@acd\or
			\sq@abd\sq@bcd
			\else\fi
			
		\end{scope}
	\end{scope}
}

\def\square#1{
	\pgfkeys{#1}
	\sq@{0}{(180:2)}\sq@{1}{(0:2)}
	
	\begin{scope}[commutative diagrams/.cd, every arrow, every label]
		\draw[->] [shorten >=10pt, shorten <=10pt, /square/arrowstyle/0123] (s0) --
		node [/square/labelstyle/0123] {%
			\pgfkeysvalueof{/square/label/0123}} (s1);
		
	\end{scope}
}

\makeatother

\makeatletter

\def\labelstylecode@pent#1{%
	\pgfkeys@split@path%
	\edef\label@key{/pentagon/label/\pgfkeyscurrentname}%
	\edef\style@key{\pgfkeyscurrentkey/.@val}%
	\def\temp@a{#1}%
	\def\temp@b{\pgfkeysnovalue}%
	\ifx\temp@a\temp@b
	\pgfkeysgetvalue{\label@key}\temp@a
	\ifx\temp@a\temp@b\else
	\pgfkeysalso{commutative diagrams/.cd, \style@key}%
	\fi
	\else
	\pgfkeys{\style@key/.code = \pgfkeysalso{#1}}%
	\fi}
\def\arrowstylecode@pent#1{%
	\edef\style@key{\pgfkeyscurrentkey/.@val}%
	\def\temp@a{#1}%
	\def\temp@b{\pgfkeysnovalue}%
	\ifx\temp@a\temp@b
	\pgfkeysalso{commutative diagrams/.cd, \style@key}%
	\else
	\pgfkeys{\style@key/.code = \pgfkeysalso{#1}}%
	\fi}

\pgfkeys{
	/pentagon/label/.cd,
	0/.initial = {$\bullet$}, 1/.initial = {$\bullet$}, 2/.initial = {$\bullet$},
	3/.initial = {$\bullet$}, 4/.initial = {$\bullet$},
	01/.initial, 12/.initial, 23/.initial, 34/.initial, 04/.initial,
	02/.initial, 03/.initial, 13/.initial, 14/.initial, 24/.initial,
	012/.initial, 013/.initial, 014/.initial, 023/.initial, 024/.initial,
	034/.initial, 123/.initial, 124/.initial, 134/.initial, 234/.initial,
	0123/.initial, 0124/.initial, 0134/.initial, 0234/.initial, 1234/.initial,
	01234/.initial,
	/pentagon/labelstyle/.cd,
	01/.@val/.initial, 12/.@val/.initial, 23/.@val/.initial, 34/.@val/.initial,
	04/.@val/.initial=\pgfkeysalso{swap},
	02/.@val/.initial=\pgfkeysalso{description},
	03/.@val/.initial=\pgfkeysalso{description},
	13/.@val/.initial=\pgfkeysalso{description},
	14/.@val/.initial=\pgfkeysalso{description},
	24/.@val/.initial=\pgfkeysalso{description},
	012/.@val/.initial=\pgfkeysalso{swap, above},
	013/.@val/.initial=\pgfkeysalso{below = 1pt},
	014/.@val/.initial=\pgfkeysalso{below},
	023/.@val/.initial=\pgfkeysalso{swap, above = 1pt},
	024/.@val/.initial=\pgfkeysalso{below left = -1pt and -1pt},
	034/.@val/.initial=\pgfkeysalso{swap, right},
	123/.@val/.initial=\pgfkeysalso{below left = -1pt and -1pt},
	124/.@val/.initial=\pgfkeysalso{swap, right},
	134/.@val/.initial=\pgfkeysalso{left},
	234/.@val/.initial=\pgfkeysalso{swap, below right = -1pt and -1pt},
	0123/.@val/.initial,
	0124/.@val/.initial=\pgfkeysalso{swap},
	0134/.@val/.initial,
	0234/.@val/.initial=\pgfkeysalso{swap},
	1234/.@val/.initial,
	01234/.@val/.initial,
	01/.code=\labelstylecode@pent{#1}, 02/.code=\labelstylecode@pent{#1},
	03/.code=\labelstylecode@pent{#1}, 04/.code=\labelstylecode@pent{#1},
	12/.code=\labelstylecode@pent{#1}, 13/.code=\labelstylecode@pent{#1},
	14/.code=\labelstylecode@pent{#1}, 23/.code=\labelstylecode@pent{#1},
	24/.code=\labelstylecode@pent{#1}, 34/.code=\labelstylecode@pent{#1},
	012/.code=\labelstylecode@pent{#1}, 013/.code=\labelstylecode@pent{#1},
	014/.code=\labelstylecode@pent{#1}, 023/.code=\labelstylecode@pent{#1},
	024/.code=\labelstylecode@pent{#1}, 034/.code=\labelstylecode@pent{#1},
	123/.code=\labelstylecode@pent{#1}, 124/.code=\labelstylecode@pent{#1},
	134/.code=\labelstylecode@pent{#1}, 234/.code=\labelstylecode@pent{#1},
	0123/.code=\labelstylecode@pent{#1}, 0124/.code=\labelstylecode@pent{#1},
	0134/.code=\labelstylecode@pent{#1}, 0234/.code=\labelstylecode@pent{#1},
	1234/.code=\labelstylecode@pent{#1}, 01234/.code=\labelstylecode@pent{#1},
	/pentagon/arrowstyle/.cd,
	01/.@val/.initial, 12/.@val/.initial, 23/.@val/.initial, 34/.@val/.initial,
	04/.@val/.initial, 02/.@val/.initial, 03/.@val/.initial, 13/.@val/.initial,
	14/.@val/.initial, 24/.@val/.initial,
	012/.@val/.initial=\pgfkeysalso{Rightarrow},
	013/.@val/.initial=\pgfkeysalso{Rightarrow},
	014/.@val/.initial=\pgfkeysalso{Rightarrow},
	023/.@val/.initial=\pgfkeysalso{Rightarrow},
	024/.@val/.initial=\pgfkeysalso{Rightarrow},
	034/.@val/.initial=\pgfkeysalso{Rightarrow},
	123/.@val/.initial=\pgfkeysalso{Rightarrow},
	124/.@val/.initial=\pgfkeysalso{Rightarrow},
	134/.@val/.initial=\pgfkeysalso{Rightarrow},
	234/.@val/.initial=\pgfkeysalso{Rightarrow},
	0123/.@val/.initial=\pgfkeysalso{triple},
	0124/.@val/.initial=\pgfkeysalso{triple},
	0134/.@val/.initial=\pgfkeysalso{triple},
	0234/.@val/.initial=\pgfkeysalso{triple},
	1234/.@val/.initial=\pgfkeysalso{triple},
	01234/.@val/.initial=\pgfkeysalso{quadruple},
	01/.code=\arrowstylecode@pent{#1}, 02/.code=\arrowstylecode@pent{#1},
	03/.code=\arrowstylecode@pent{#1}, 04/.code=\arrowstylecode@pent{#1},
	12/.code=\arrowstylecode@pent{#1}, 13/.code=\arrowstylecode@pent{#1},
	14/.code=\arrowstylecode@pent{#1}, 23/.code=\arrowstylecode@pent{#1},
	24/.code=\arrowstylecode@pent{#1}, 34/.code=\arrowstylecode@pent{#1},
	012/.code=\arrowstylecode@pent{#1}, 013/.code=\arrowstylecode@pent{#1},
	014/.code=\arrowstylecode@pent{#1}, 023/.code=\arrowstylecode@pent{#1},
	024/.code=\arrowstylecode@pent{#1}, 034/.code=\arrowstylecode@pent{#1},
	123/.code=\arrowstylecode@pent{#1}, 124/.code=\arrowstylecode@pent{#1},
	134/.code=\arrowstylecode@pent{#1}, 234/.code=\arrowstylecode@pent{#1},
	0123/.code=\arrowstylecode@pent{#1}, 0124/.code=\arrowstylecode@pent{#1},
	0134/.code=\arrowstylecode@pent{#1}, 0234/.code=\arrowstylecode@pent{#1},
	1234/.code=\arrowstylecode@pent{#1}, 01234/.code=\arrowstylecode@pent{#1}
}

\def\pent@abc{%
	\draw [/pentagon/arrowstyle/012] (198:0.45) --
	node [/pentagon/labelstyle/012] {
		\pgfkeysvalueof{/pentagon/label/012}} (198:0.8);
}
\def\pent@bcd{%
	\draw [/pentagon/arrowstyle/123] (126:0.45) --
	node [/pentagon/labelstyle/123] {
		\pgfkeysvalueof{/pentagon/label/123}} (126:0.8);
}
\def\pent@cde{%
	\draw [/pentagon/arrowstyle/234] (54:0.45) --
	node [/pentagon/labelstyle/234] {
		\pgfkeysvalueof{/pentagon/label/234}} (54:0.8);
}
\def\pent@ade{%
	\draw [/pentagon/arrowstyle/034] (-40:0.6) --
	node [/pentagon/labelstyle/034] {
		\pgfkeysvalueof{/pentagon/label/034}} (-5:0.5);
}
\def\pent@abe{
	\draw [/pentagon/arrowstyle/014] (-70:0.55) --
	node [/pentagon/labelstyle/014] {
		\pgfkeysvalueof{/pentagon/label/014}} (-110:0.55);
}
\def\pent@acd{%
	\draw [/pentagon/arrowstyle/023] (55:0.3) --
	node [/pentagon/labelstyle/023] {
		\pgfkeysvalueof{/pentagon/label/023}} (125:0.3);
}
\def\pent@bde{%
	\draw [/pentagon/arrowstyle/134] (-5:0.4) --
	node [/pentagon/labelstyle/134] {
		\pgfkeysvalueof{/pentagon/label/134}} (35:0.5);
}
\def\pent@ace{%
	\draw [/pentagon/arrowstyle/024] (-45:0.45) --
	node [/pentagon/labelstyle/024] {
		\pgfkeysvalueof{/pentagon/label/024}} (-45:0.1);
}
\def\pent@abd{%
	\draw [/pentagon/arrowstyle/013] (-90:0.22) --
	node [/pentagon/labelstyle/013] {
		\pgfkeysvalueof{/pentagon/label/013}} (-150:0.46);
}
\def\pent@bce{%
	\draw [/pentagon/arrowstyle/124] (188:0.4) --
	node [/pentagon/labelstyle/124] {
		\pgfkeysvalueof{/pentagon/label/124}} (150:0.55);
}

\def\pent@#1#2{
	\begin{scope}[shift=#2, commutative diagrams/every diagram]
		
		\foreach \i in {0,1,2,3,4} {
			\tikzmath{\a = 270 - (72 * \i);}
			\node (n{#1}\i) at (\a:1) {
				\pgfkeysvalueof{/pentagon/label/\i}};
		}
		
		\node (p#1) at (0,0) [circle, inner sep = 0pt,
		fit = (n{#1}0.center)(n{#1}1.center)(n{#1}2.center)
		(n{#1}3.center)(n{#1}4.center)] {};
		
		\begin{scope}[commutative diagrams/.cd, every arrow, every label]
			\ifcase #1
			\def\list{0/1, 1/2, 2/3, 3/4, 0/4, 0/2, 0/3}\or
			\def\list{0/1, 1/2, 2/3, 3/4, 0/4, 1/3, 1/4}\or
			\def\list{0/1, 1/2, 2/3, 3/4, 0/4, 0/2, 2/4}\or
			\def\list{0/1, 1/2, 2/3, 3/4, 0/4, 0/3, 1/3}\or
			\def\list{0/1, 1/2, 2/3, 3/4, 0/4, 1/4, 2/4}\else
			\def\list{}\fi
			
			\foreach \s / \e in \list {
				\draw [/pentagon/arrowstyle/\s\e] (n{#1}\s) --
				node [/pentagon/labelstyle/\s\e] {
					\pgfkeysvalueof{/pentagon/label/\s\e}} (n{#1}\e);
			}
			
			\ifcase #1
			\pent@abc\pent@acd\pent@ade\or
			\pent@bcd\pent@bde\pent@abe\or
			\pent@cde\pent@ace\pent@abc\or
			\pent@ade\pent@abd\pent@bcd\or
			\pent@abe\pent@bce\pent@cde
			\else\fi
			
		\end{scope}
	\end{scope}
}

\def\pentagon#1{
	\pgfkeys{#1}
	\pent@{2}{(270:3)}\pent@{0}{(198:3)}\pent@{3}{(126:3)}
	\pent@{1}{(54:3)}\pent@{4}{(342:3)}
	
	\begin{scope}[commutative diagrams/.cd, every arrow, every label]
		\draw [/pentagon/arrowstyle/0123] (p0) --
		node [/pentagon/labelstyle/0123] {
			\pgfkeysvalueof{/pentagon/label/0123}} (p3);
		
		\draw [/pentagon/arrowstyle/0134] (p3) --
		node [/pentagon/labelstyle/0134] {
			\pgfkeysvalueof{/pentagon/label/0134}} (p1);
		
		\draw [/pentagon/arrowstyle/1234] (p1) --
		node [/pentagon/labelstyle/1234] {
			\pgfkeysvalueof{/pentagon/label/1234}} (p4);
		
		\draw [/pentagon/arrowstyle/0234] (p0) --
		node [/pentagon/labelstyle/0234] {
			\pgfkeysvalueof{/pentagon/label/0234}} (p2);
		
		\draw [/pentagon/arrowstyle/0124] (p2) --
		node [/pentagon/labelstyle/0124] {
			\pgfkeysvalueof{/pentagon/label/0124}} (p4);
		
		\draw [/pentagon/arrowstyle/01234] (270:0.75) --
		node [/pentagon/labelstyle/01234] {
			\pgfkeysvalueof{/pentagon/label/01234}} (90:0.75);
	\end{scope}
}

\makeatother
\newcommand{\plus}[1]{\mathop{\amalg}\limits_{#1}}
\newcommand{\st}{\underline{\mathrm{Strat}}}

\newcommand{\C}{\mathcal{C}}
\newcommand{\D}{\mathcal{D}}
\newcommand{\E}{\mathcal{E}}

\newcommand{\V}{\mathcal{V}}

\newcommand{\xto}{\xrightarrow}
\newcommand{\Fl}{\text{Cell}}
\newcommand{\On}[1]{\mathcal{O}_{#1}}

\def\vphi{\varphi}
\def\lrar{\rightarrow}
\def\llar{\longleftarrow}

\usepackage[font=small,format=hang,labelfont={sf,bf}]{caption}
\newcommand{\oplax}{\mathrm{oplax}}
\newcommand{\minisimeq}{\scalebox{0.5}{\ensuremath\simeq}}
\newcommand{\cFun}{\mathrm{Fun}^{\minisimeq}}

\setlist[itemize]{leftmargin=*}
\setlist[enumerate]{leftmargin=*}

\makeatletter

\renewcommand{\tocsection}[3]{%
\indentlabel{\@ifnotempty{#2}{\parbox[b]{3ex}{\bfseries\ignorespaces#1 #2}}}\bfseries#3} 

\renewcommand{\tocsubsection}[3]{%
\indentlabel{\@ifnotempty{#2}{\hspace{1.6em}\parbox[b]{5ex}{\ignorespaces#1 #2}}}#3}

\makeatother

\title{Gray tensor products and lax functors of \((\infty,2)\)-categories}
\author{Andrea Gagna}
\address{Universita Karlova\\ Matematicko-fyzik\'{a}ln\'{i} fakulta \\ Matematická sekce \\ Sokolovská 83 \\ 186 75 Praha 8 \\ Czech Republic}
\email{gagna@karlin.mff.cuni.cz}
\urladdr{https://sites.google.com/view/andreagagna/home}
\author{Yonatan Harpaz}
\address{Institut Galilée\\ Université Paris 13\\ 99 avenue Jean-Baptiste Clément\\ 93430 Villeta-neuse\\ France}
\email{harpaz@math.univ-paris13.fr}
\urladdr{https://www.math.univ-paris13.fr/~harpaz}
\author{Edoardo Lanari}
\address{Institute of Mathematics CAS \\ \v{Z}itn\'a 25 \\115 67   Praha 1\\ Czech Republic}
\email{edoardo.lanari.el@gmail.com}
\urladdr{https://sites.google.com/view/edoardo-lanari}
\subjclass[2010]{18G30, 18G55, 55U10, 55U35}
\begin{document}
\begin{abstract}
	We give a definition of the Gray tensor product in the setting of scaled simplicial sets which is associative and forms a left Quillen bifunctor with respect to the bicategorical model category of Lurie. We then introduce a notion of oplax functor in this setting, and use it in order to characterize the Gray tensor product by means of a universal property. A similar characterization was used by Gaitsgory and Rozenblyum in their definition of the Gray product, thus giving a promising lead for comparing the two settings.
\end{abstract}
\maketitle
\tableofcontents
\section*{Acknowledgements}
The first author is supported by GA\v{C}R EXPRO 19-28628X.
The third author gratefully acknowledges the support of Praemium Academiae of M.~Markl and RVO:67985840, as well as fruitful conversations with Nick Rozenblyum during his stay at MSRI. 

\section*{Introduction}

The theory of 2-categories provides a powerful framework
in which one can develop formal category theory, allowing for notions such as adjunctions, Kan extensions, correspondences and lax (co)limits to be studied in an abstract setting.
The category of 2-categories carries a monoidal structure, given by the Gray tensor product, which makes use of the 2-dimensional structure available.
This monoidal structure comes in a few flavors, as is often the case in the 2-categorical setting:
there is a pseudo version,
which was the first one to be defined by Gray in~\cite{GrayFormal} (hence the name), as well as a lax and an oplax version.
The Gray tensor product serves the purpose of representing pseudo/lax natural transformations, and it is thus a replacement for the cartesian product of 2-categories when one considers these weaker notion of morphisms. Moreover, while the cartesian product
is badly-behaved with respect to the folk model category
structure on \(2\)-categories, Lack~\cite{LackModelBicat} showed that the pseudo Gray tensor is compatible with
the folk model category, and recently Ara and Lucas~\cite{AraLucas} proved that the same holds for
the lax and oplax version of the Gray tensor product.

In the last decade, the study of \(\infty\)-categories
as a building block for homotopy-coherent mathematics
has contributed to the development of spectral and
derived algebraic geometry, culminating in impressive applications, such as the proof by Gaitsgory and Lurie in~\cite{GaitsgoryLurieWeil} of Weil's conjecture on Tamagawa numbers for function fields.
In the same spirit of organizing ordinary category theory from a 2-dimensional perspective, the framework of \((\infty, 2)\)-categories allows to better understand homotopy coherent constructions performed on \((\infty,1)\)-categories.

In this paper, we introduce and study a particularly well-behaved model of
the (oplax) Gray tensor product for \((\infty, 2)\)-categories.
We work in the category of 
scaled simplicial sets equipped with the bicategorical model structure constructed 
in~\cite{LurieGoodwillie}, and further developed in~\cite{Equivalence},
and we provide a Gray tensor product which is associative on-the-nose and is moreover a left Quillen bifunctor. In particular, with the bicategorical model structure the category of scaled simplicial sets forms a (non-symmetric) monoidal model category.  This implies, for example, that the Gray tensor product preserve homotopy colimits in each variable separately, and that the underlying \(\infty\)-category \(\Cat_{(\infty,2)}\) of \((\infty,2)\)-categories carries the structure of a presentably monoidal \(\infty\)-category with respect to the Gray product.

Versions of the Gray tensor product in different contexts already appear in the literature. Verity~\cite{VerityWeakComplicialI} defines a Gray tensor product in stratified simplicial sets compatible with the complicial model category structure for~\((\infty, \infty)\)-categories. This can be truncated to be compatible with the model category structure for \(2\)-truncated saturated complicial sets established in~\cite{OzornovaRovelliNComplicial}, giving a Gray tensor product for this model of~\((\infty, 2)\)-categories. In this paper, we prove that this version of the Gray tensor is equivalent to ours, via the equivalence established in \cite{Equivalence}. Another approach has been recently provided by Maehara~\cite{MaeharaGray} via the combinatorics of \(\Theta_2\)-sets.

The main motivation behind our interest in the Gray tensor product
comes from the recent book~\cite{GaitsgoryRozenblyumStudy},
where Gaitsgory and Rozenblyum develop a formalism of categories of correspondences which makes use of the Gray tensor product.
However, there are some unproven claims in the technical section
dealing with \((\infty, 2)\)-categories, and one of the aims of this project is to provide a proof
for some of these statements in the framework of scaled simplicial sets. For example,
the preservation of (homotopy) colimits in each variable provides a reference for~\cite[Propositions~10.3.2.6 and~10.3.2.9]{GaitsgoryRozenblyumStudy} in the setting of scaled simplicial sets. 

To fully achieve the goal above it is however necessary to compare our Gray tensor product with that of Gaitsgory--Rozenblyum. While we do not provide such a comparison in the current paper, we pave the way towards one by establishing a universal property for our Gray product in terms of oplax functors of \(\infty\)-bicategories. This mirrors the approach of Gaitsgory--Rozenblyum for the definition of the Gray product, and reduces the comparison of the respective Gray products that that of the two notions of oplax functors. The latter comparison is the subject of current work in progress and we hope to settle this question in a future paper.

It is worth pointing out that, despite the existence of various models of Gray tensor product, none of them have been proven to satisfy the characterizing universal property formulated in~\cite[Chapter 10, Point 3.2.3]{GaitsgoryRozenblyumStudy}, and in particular non were ever shown to be equivalent to that of loc.\ cit. This has been one obstacle in providing proofs for some of the unproven claims in~\cite[Chapter 10]{GaitsgoryRozenblyumStudy}, and we hope and expect that the present paper will open a path leading to such proofs. 

The present paper is organized as follows.
In \S\ref{sec:preliminaries} we fix the notation and recall the categories of marked and scaled simplicial sets, with their respective marked categorical and bicategorical model structures. 
We then examine the mapping spaces between two scaled simplicial sets following~\cite{LurieGoodwillie} and we derive
some standard consequences.

In the \S\ref{sec:gray-lax} we introduce the Gray tensor product
and the closely related notion of (op)lax transformations, and study some of its basic properties. We proceed in \S\ref{sec:comparison} to compare our definition with that of Verity in the setting of complicial sets, under the Quillen equivalence between the two models established in~\cite{Equivalence}. Then, in \S\ref{sec:gray-is-quillen} we prove that our Gray tensor product is left Quillen bifunctor, which constitutes the main result of this paper. 

The final~\S\ref{sec:oplax} is dedicated to the establishment of a universal mapping property for the Gray product in terms of a suitable notion of oplax functors, which we generalize from classical 2-category theory to the setting of \(\infty\)-bicategories. This universal characterization allows one in principle to recognize the Gray tensor product in any given model for \((\infty,2)\)-categories, as soon as this model supports a notion of oplax functors. Our main motivation here is to initiate a path towards the comparison of the present Gray tensor product with that of Gaitsgory--Rozenblyum, so that properties proven on the Gray tensor product at hand (including those derived from the results of the present paper, such as associativity and compatibility with colimits), could be applied to that of Gaitsgory--Rozenblyum, thus providing proofs for the unproven claims listed in~\cite[10.0.4.2]{GaitsgoryRozenblyumStudy}.

\section{Preliminaries}\label{sec:preliminaries}
In this section we recall the necessary definitions and theorems that will be used throughout this paper.

\begin{notate}
We will denote by \(\Delta\) the category of simplices, that is
the category whose objects are the finite non-empty ordinals \([n] = \{0, 1, 2, \dots, n\}\)
and morphisms are the non-decreasing maps.
We will denote by \(\s\) the category of simplicial sets, that is the category of
presheaves on sets of~\(\Delta\), and will employ the standard notation
\(\Del^n\) for the \(n\)-simplex, \ie the simplicial set representing
the object~\([n]\) of~\(\Delta\).
For any subset \(\emptyset \neq S \subseteq [n]\) we will write \(\Del^S \subseteq \Del^n\)
to denote the \((|S|-1)\)-dimensional face of \(\Del^n\) whose vertices belong to \(S\).
For \(0 \leq i \leq n\) we will denote by \(\Lam^n_i \subseteq \Del^n\)
the \(i\)-th horn in \(\Del^n\), that is, the subsimplicial set of \(\Del^n\) spanned
by all the \((n-1)\)-dimensional faces containing the \(i\)-th vertex. 
For any simplicial set \(X\) and any integer \(p\geq 0\), we will denote by \(\deg_p(X)\)
the set of degenerate~\(p\)-simplices of~\(X\).
\end{notate}

By \ndef{\(\infty\)\nbd-category} we will always mean a \emph{quasi-category},
\ie a simplicial set \(X\) which admits extensions for all inclusions 
\(\Lambda^n_i\rightarrow\Delta^n\), for all \(n > 1\) and all \(0 < i < n\)
(known as \ndef{inner horns}). Given an \(\infty\)-category \(X\),
we will denote its homotopy category by \(\ho(X)\).
This is the ordinary category having as objects the \(0\)-simplices of \(X\),
and as morphisms \(x \rightarrow y\) the set of equivalence classes of \(1\)-simplices
\(f\colon x \rightarrow y\) of \(X\) under the equivalence relation generated by identifying
\(f\) and \(f'\) if there is a \(2\)-simplex \(H\) of \(X\) with
\( H_{|\Del^{\{1,2\}}}=f, \ H_{|\Del^{\{0,2\}}}=f'\) and \(H_{|\Del^{\{0,1\}}}\) degenerate on~\(x\).
We recall that the functor \(\ho\colon \nCat{\infty}\rightarrow \nCat{1}\)
is left adjoint to the ordinary nerve functor \(N \colon \nCat{1} \to \nCat{\infty}\).

\subsection{Marked simplicial sets and \(\infty\)-categories}

\begin{define}
	A \emph{marked simplicial set} is a pair \((X,E)\) where \(X\) is
	simplicial set and \(E\) is a subset of the set of \(1\)-simplices of \(X\),
	called \emph{marked simplices} or \ndef{marked edges}, containing the degenerate ones.
	A map of marked simplicial sets \(f\colon (X,E_X)\rightarrow (Y,E_Y)\) is a map of simplicial sets \(f\colon X \rightarrow Y\) satisfying \(f(E_X)\subseteq E_Y\).
	
	The category of marked simplicial sets will be denoted by \(\s^+\). 
\end{define}

\begin{rem}\label{rem:marked_loc_cartesian_closed}
	The category \(\s^+\) of marked simplicial sets
	admits an alternative description, as the category of models of a limit sketch.
	In particular, it is a reflective localization of a presheaf category
	and it is a cartesian closed category.
	In fact, it is a locally cartesian closed category.
\end{rem}

\begin{thm}[\cite{HTT}]\label{thm:marked-categorical}
	There exists a model category structure on the category \(\s^+\)
	of marked simplicial sets in which cofibrations are exactly the monomorphisms
	and the fibrant objects are marked simplicial sets \((X, E)\) in which \(X\)
	is an \(\infty\)-category and \(E\) is the set of equivalences of \(X\),
	\ie \(1\)-simplices \(f\colon \Delta^1 \rightarrow X\) which are invertible in \(\ho(X)\). 
\end{thm}

This is a special case of Proposition 3.1.3.7 in \cite{HTT}, when \(S=\Delta^0\).
We will refer to the model structure of Theorem~\ref{thm:marked-categorical}
as the \ndef{marked categorical model structure},
and its weak equivalences as \ndef{marked categorical equivalences}.

\begin{rem}
	Marked simplicial sets endowed with the marked categorical model structure
	are a model for \((\infty,1)\)-categories. 
\end{rem}

\subsection{Scaled simplicial sets and \(\infty\)-bicategories}\label{sec:scaled}

\begin{define}[\cite{LurieGoodwillie}]
	A \emph{scaled simplicial set} is a pair \((X,T)\) where \(X\) is simplicial set and \(T\) is a subset of the set of 2-simplices of \(X\), called \emph{thin \(2\)-simplices} or \emph{thin triangles}, containing the degenerate ones. A map of scaled simplicial sets \(f\colon (X,T_X)\rightarrow (Y,T_Y)\) is a map of simplicial sets \(f\colon X \rightarrow Y\) satisfying \(f(T_X)\subseteq T_Y\). We will often refer to a scaled simplicial set just without explicitly mentioning its thin 2-simplices, when this causes no ambiguity.
	
	We will denote by \(\Ss\) the category of scaled simplicial sets.
\end{define}

\begin{notate}\label{not:scaled_flat-sharp}
	Let \(X\) be a simplicial set. We will denote by \(X_{\flat} = (X, \deg_2(X))\)
	the scaled simplicial set where the
	thin triangles of \(X\) are the degenerate \(2\)-simplices and by \(X_{\sharp} = (X, X_2)\) the scaled simplicial
	set where all the triangles of \(X\) are thin.
	The assignments
	\[X \mapsto X_\flat\qquad\text{and}\qquad X \mapsto X_\sharp\]
	are the left and right adjoint of the obvious forgetful functor
	\(\Ss \to \s\).
\end{notate}

\begin{rem}\label{rem:scaled_reflection}
	The category \(\Ss\) admits an alternative description, as the category of models of a limit sketch.
	In particular, it is a reflective localization of a presheaf category and so it is cartesian closed.
	In fact, we can consider the category \(\Delta_{\sca}\) having \(\{[k]\}_{k\geq 0}\cup \{[2]_t\}\) as set of objects,
	obtained from \(\Delta\)
	by adding an extra object and maps \([2]\rightarrow[2]_t, \ \sigma^i_t\colon [2]_t\rightarrow [1]\)
	for \(i=0,1\) satisfying the obvious relations.
	The category \(\Ss\) is then the reflective localization of the category of presheaves \(\Psh(\Delta_{\sca})\) 
	(of sets) at the arrow \([2]_t\amalg_{[2]}[2]_t \rightarrow [2]_t  \),
	where we have identified an object of \(\Delta_{\sca}\) with its corresponding representable presheaf.
	Equivalently, the local objects are those presheaves \(X\colon\Delta_{\sca}^{\mathrm{op}}\rightarrow\mathbf{Set}\)
	for which \(X([2]_t)\rightarrow X([2])\) is a monomorphism.
\end{rem}

\begin{notate}\label{not:scaled_subset}
	We will often speak only of the non-degenerate thin \(2\)-simplices
	when considering a scaled simplicial set. For example, if \(X\) is a simplicial set
	and \(T\) is any set of triangles in \(X\) then we will denote by \((X,T)\)
	the scaled simplicial set whose underlying simplicial set is \(X\) and
	whose thin triangles are \(T\) together with the degenerate triangles.
	If \(L \subseteq K\) is a subsimplicial set then we use \(T|_L := T \cap L_2\)
	to denote the set of triangles in \(L\) whose image in \(K\) is contained in \(T\). 
\end{notate}

The following set of inclusions will characterize the model structure on scaled simplicial sets for \(\infty\)-bicategories. We are implicitly making use of Corollary~3.0.6 of~\cite{Equivalence} in that we assume weak \(\infty\)-bicategories coincide with \(\infty\)-bicategories, as defined in~\cite{LurieGoodwillie}.

\begin{define}
	\label{d:anodyne}
	Let \(\bS\) be the set of maps of scaled simplicial sets consisting of:
	\begin{enumerate}
		\item[(i)] the inner horns inclusions
		\[
		 \bigl(\Lambda^n_i,\Delta^{\{i-1,i,i+1\}}\bigr)\rightarrow \bigl(\Delta^n,\Delta^{\{i-1,i,i+1\}}\bigr)
		 \quad , \quad n \geq 2 \quad , \quad 0 < i < n ;
		\]
		\item[(ii)]\label{i:saturation} the map 
		\[
		 (\Delta^4,T)\rightarrow (\Delta^4,T\cup \{\Delta^{\{0,3,4\}}, \ \Delta^{\{0,1,4\}}\}),
		\]
		where we define
		\[
		 T\overset{\text{def}}{=}\{\Delta^{\{0,2,4\}}, \ \Delta^{\{ 1,2,3\}}, \ \Delta^{\{0,1,3\}}, \ \Delta^{\{1,3,4\}}, \ \Delta^{\{0,1,2\}}\};
		\]
		\item[(iii)] the set of maps
		\[
		\Bigl(\Lambda^n_0\plus{\Delta^{\{0,1\}}}\Delta^0,\Delta^{\{0,1,n\}}\Bigr)\rightarrow \Bigl(\Delta^n\plus{\Delta^{\{0,1\}}}\Delta^0,\Delta^{\{0,1,n\}}\Bigr)
		\quad , \quad n\geq 3.
		\]
	\end{enumerate}
\end{define}

We call \(\bS\) the set of \emph{generating anodyne morphisms}. We make use of it in the following definition.

\begin{define}
	An \emph{\(\infty\)-bicategory} is a scaled simplicial set \((X,T)\) which admits extensions along all maps in \(\bS\). Diagrammatically, if \(i\colon K \rightarrow L\) is a map in \(\bS\), then for every map \(f\colon K \rightarrow (X,T)\) there exists an extension as displayed below by the dashed arrow in the following diagram:
	\[
		\begin{tikzcd}
		K \ar[r,"f"] \ar[d,"i"{swap}]& (X,T)\\
		L \ar[ur,dashed]
		\end{tikzcd}.
	\]
\end{define}

Putting together the results in \cite{LurieGoodwillie} and \cite{Equivalence}, we get the following theorem.

\begin{thm}
	There exists a model structure on the category of scaled simplicial sets
	whose cofibrations are the monomorphisms and whose fibrant objects are the \(\infty\)-bicategories.
\end{thm}

This model structure is proved in \cite{Equivalence} to be Quillen equivalent
to Verity's one on stratified sets for saturated \(2\)-trivial complicial sets,
as defined in~\cite{VerityWeakComplicialI} and~\cite{OzornovaRovelliNComplicial}.
We depict the Quillen equivalence as follows:
\begin{equation}
\label{equiv}
\xymatrixcolsep{1pc}
\vcenter{\hbox{\xymatrix{
			**[l]\Ss \xtwocell[r]{}_{U}^{\iota}{'\perp}& **[r] \st_{2}}}} ,
\end{equation}
where \(U\) denotes the forgetful functor and \(\iota\) is an inclusion. 
Furthermore, Lurie shows in~\cite{LurieGoodwillie} that there is a \ndef{scaled homotopy coherente nerve}
\[
 \Nsc \colon \nCat{\s^+} \longrightarrow \Ss
\]
which is a Quillen equivalence, where the category \(\nCat{\s^+}\) of category enriched
in marked simplicial sets is endowed with the \(\s^+\)-enriched model category structure
(see~\cite[\S A.3.2]{HTT}).

\begin{define}
Let \((X,T)\) be a scaled simplicial set. We will say that the collection of triangles \(T\) is \ndef{saturated} if both \((X,T)\) and \((X^{\op},T)\) satisfy the extension property with respect to the~\hyperref[i:saturation]{(ii)} of Definition~\ref{d:anodyne}.
\end{define}

\begin{example}
Any \(\infty\)-bicategory is saturated.
\end{example}

\begin{rem}
It follows from~\cite[Remark 3.1.4]{LurieGoodwillie} that if a scaled simplicial set \(X\) is saturated then \(X\) satisfies the extension property with respect to the maps \((\Del^3,T_i) \to \Del^3_{\sharp}\) for \(i=1,2\) where \(T_i\) is the collection of all triangles in \(\Del^3\) except \(\Del^{\{0,i,3\}}\).
\end{rem}

\begin{define}
Let \((X,T)\) be a scaled simplicial set. We define the \ndef{saturated closure} of \(T\) to be the smallest saturated set of triangles \(T'\) which contains \(T\) (note that such a set exists because the collection of saturated sets of triangles is closed under intersection).
\end{define}

\begin{lemma}
Let \((X,T)\) be a scaled simplicial set and \(\ovl{T}\) the saturated closure of \(T\). Then the map \((X,T) \to (X,\ovl{T})\) is a trivial cofibration in \(\Ss\).
\end{lemma}
\begin{proof}
Applying the small object argument with respect to the set \(S\) consisting of the map~\hyperref[i:saturation]{(ii)} of Definition~\ref{d:anodyne} and its opposite we may find a map \((X,T) \to (X',T')\) which is a retract of transfinite compositions of pushouts of maps in \(S\) and such that \(T'\) is saturated in \(X'\). Since the maps in \(S\) are isomorphisms on the underlying simplicial set the same holds for \(X \to X'\), and so we may identify \(T'\) with a set of triangles in \(X\). Since every map in \(S\) is a trivial cofibration the same holds for the resulting map \((X,T) \to (X,T')\). Now since \(T'\) is saturated and contains \(T\) it also contains its saturated closure \(\ovl{T}\). On the other hand, since \((X,\ovl{T})\) satisfies the extension property with respect to \(S\) it also satisfies the extension property with respect to \((X,T) \to (X,T')\), and so \(T'\subseteq \ovl{T}\). We may then conclude that \(T'=\ovl{T}\) is the saturated closure of \(T\) and hence the desired result follows.
\end{proof}

\begin{define}
	\label{core defi}
	Given a scaled simplicial set \(X\), we define its \emph{core} to be the simplicial set \(X^{\thi}\) spanned by those \(n\)-simplices of \(X\) whose 2-dimensional faces are thin triangles.
\end{define}

\begin{rem}
	Notice that \(X\) and \(X^{\thi}\) agree on the 1-skeleton.
	Moreover, whenever \(X\) is an \(\infty\)-bicategory, its
	core \(X^{\thi}\) (formally, its underlying simplicial set)
	is an \(\infty\)-category.
\end{rem}

\begin{define}\label{d:equivalence}
	Let \(\C\) be an \(\infty\)-bicategory. We will say that an edge in \(\C\) is \ndef{invertible} if it is invertible when considered in the \(\infty\)-category \(\C^{\thi}\), that is, if its image in the homotopy category of \(\C^{\thi}\) is an isomorphism. We will sometimes refer to invertible edges in \(\C\) as \ndef{equivalences}. We will denote by \(\C^{\minisimeq} \subseteq \C^{\thi}\) the subsimplicial set spanned by the invertible edges. Then \(\C^{\minisimeq}\) is an \(\infty\)-groupoid (that is, a Kan complex), which we call the \emph{core groupoid} of \(\C\). It can be considered as the \(\infty\)-groupoid obtained from \(\C\) by discarding all non-invertible \(1\)-cells and \(2\)-cells. If \((X,T)\) is an arbitrary scaled simplicial set then we will say that an edge in \(X\) is invertible if its image in \(\C\) is invertible for any bicategorical equivalence \((X,T) \to \C\). This does not depend on the choice of \(\C\).
\end{define}

\begin{notate}
	\label{n:mapping}
	Let \(\C\) be an \(\infty\)-bicategory and let \(x,y \in \C\) be two vertices.
	In section~4.2 of~\cite{LurieGoodwillie}, Lurie gives an explicit
	model for the mapping \(\infty\)\nbd-cat\-egory  from \(x\) to \(y\) in \(\C\) that we now recall.
	Let \(\Hom_{\C}(x,y)\) be the marked simplicial set whose \(n\)-simplices are given by maps \(f\colon\Del^n \times \Del^1 \lrar \C\) such that \(f_{|\Del^n \times \{0\}}\) is constant on \(x\), \(f_{|\Del^n \times \{1\}}\) is constant on \(y\), and the triangle \(f_{|\Del^{\{(i,0),(i,1),(j,1)\}}}\) is thin 
	for every \(0 \leq i\leq j \leq n\). An edge \(f\colon\Del^1 \times \Del^1 \lrar \C\) of \(\Hom_{\C}(x,y)\) is marked exactly when the triangle \(f_{|\Del^{\{(0,0),(1,0),(1,1)\}}}\) is thin. 
	The assumption that \(\C\) is an \(\infty\)-bicategory implies that 
	the marked simplicial set \(\Hom_{\C}(x,y)\) is \ndef{fibrant} in the marked categorical model structure, that is, it is an \(\infty\)-category whose marked edges are exactly the equivalences.	
\end{notate}

\begin{rem}\label{r:dwyer-kan}
	By Remark~4.2.1 and Theorem~4.2.2 of~\cite{LurieGoodwillie}, if \(\D\) is a fibrant \(\Set^+_\Del\)\nbd-en\-riched category and \(\C\) is an \(\infty\)-bicategory equipped with a bicategorical equivalence \(\vphi:\C \simeq \Nsc(\D)\) then the maps 
	\[ \Map_{\C}(x,y) \longrightarrow \Map_{\Nsc(\D)}(\vphi(x),\vphi(y)) \llar \Map_{\D}(\vphi(x),\vphi(y)) \]
	are categorical equivalences of marked simplicial sets for every pair of vertices \(x, y\) of~\(\C\). It then follows that a map \(\vphi: \C \lrar \C'\) of \(\infty\)-bicategories is a bicategorical equivalence if and only if it is essentially surjective (that is, every object in \(\C'\) is equivalent to an object in the image, see Definition~\ref{d:equivalence})
	and the induced map \(\Map_{\C}(x,y) \lrar \Map_{\C'}(\vphi(x),\vphi(y))\) is a categorical equivalence of (fibrant) marked simplicial sets for every \(x,y \in \C\).
\end{rem}

\begin{rem}\label{r:underlying}
	It follows from Remark~\ref{r:dwyer-kan} that if \(\vphi\colon\C \lrar \C'\) is a bicategorical equivalence of \(\infty\)-bicategories then the induced map \(\vphi^{\thi}:\C^{\thi} \lrar (\C')^{\thi}\) is an equivalence of \(\infty\)-categories.
\end{rem}

It is shown in~\cite[Proposition 3.1.8]{LurieGoodwillie} that the cartesian product
\[ \times:\Ss \times \Ss \longrightarrow \Ss \]
of a scaled anodyne map and a monomorphism is contained in the saturation of the class of scaled anodyne maps. Therefore, thanks to Theorem 5.1 of \cite{Equivalence}, the cartesian product is a left Quillen bifunctor with respect to the bicategorical model structure, i.e. \(\Ss\) is a cartesian closed model category. In particular, for a every two scaled simplicial sets \(X,Y\) we have a mapping object \(\Fun(X,Y)\) which satisfies (and is determined by) the exponential formula
\[ \Hom_{\Ss}(Z,\Fun(X,Y)) \cong \Hom_{\Ss}(Z \times X,Y) .\]
In addition, when \(Y\) is an \(\infty\)-bicategory the mapping object \(\Fun(X,Y)\) is an \(\infty\)\nbd-bi\-category as well, which we can consider as the \(\infty\)-bicategory of functors from \(X\) to \(Y\). In this case we will denote by \(\Fun^{\thi}(X,Y) \subseteq \Fun(X,Y)\) the core \(\infty\)-category and by \(\cFun(X,Y) \subseteq \Fun^{\thi}(X,Y)\) the core \(\infty\)-groupoid of \(\Fun(X,Y)\). In particular, \(\Fun^{\thi}(X,Y)\) is an \(\infty\)-category and \(\cFun(X,Y)\) is a Kan complex, 
which we consider as the \ndef{space of functors} from \(X\) to \(Y\). We note the following:

\begin{lemma}\label{l:test-equiv}
	Let \(f: X \lrar Y\) be a map of scaled simplicial sets. Then \(f\) is a bicategorical equivalence if and only if for every \(\infty\)-bicategory \(\C\) the induced map
	\begin{equation}\label{e:f-map} 
	f^*\colon\cFun(Y,\C) \longrightarrow \cFun(X,\C) 
	\end{equation}
	is an equivalence of Kan complexes.
\end{lemma}
\begin{proof}
	If \(f\colon X \lrar Y\) is a bicategorical equivalence then \(\Fun(Y,\C) \lrar \Fun(X,\C)\) is a bicategorical equivalence for every \(\infty\)-bicategory \(\C\) since \(\Ss\) is cartesian closed and every object is cofibrant. It then follows from Remark~\ref{r:underlying} that~\eqref{e:f-map} is an equivalence of Kan complexes. 
	
	Now suppose that~\eqref{e:f-map} is an equivalence of Kan complexes. By the argument above the property that~\eqref{e:f-map} is an equivalence of Kan complexes will hold for any arrow which is levelwise bicategorically equivalent to \(f'\) in \(\Ss\). We may hence assume without loss of generality that \(X\) and \(Y\) are fibrant, that is, they are \(\infty\)\nbd-bi\-categories. Taking \(\C=X\) in~\eqref{e:f-map} we may conclude there exists a map \(g\colon Y \lrar X\) such that \(gf\colon X \lrar X\) is in the same component as \(\Id\colon X \lrar X\) in \(\cFun(X,X)\). There is hence an arrow \(e\colon\Del^1 \lrar \cFun(X,X)\) such that \(e(0)=\Id\) and \(e(1)=gf\). Let \(K\) be a contractible Kan complex which contains a non-degenerate arrow \(\Del^1 \subseteq K\), which we can write as \(v \lrar u\). Then we can extend the map \(e \colon \Del^1 \lrar \cFun(X,X)\) to a map \(\wtl{e}\colon K \lrar \cFun(X,X)\). By adjunction we thus get a map \(H\colon X \times K \lrar X\) such that \(H_{|X \times \{v\}}=\Id\) and \(H_{|X \times \{u\}}=gf\). Since the map \(X \times K \lrar X\)
	induced by the canonical morphism
	\(K \to \Del^0\) is a bicategorical equivalence and the compositions  \(X \times \{v\} \lrar X \times K \lrar X\) and \(X \times \{u\} \lrar X \times K \lrar X\) are both the identity 
	on \(X\), we may consider \(X \times K\) as a cylinder object for \(X\) in \(\Ss\), so that \(gf\) is homotopic to the identity with respect to the bicategorical model structure. We now claim that \(fg\colon Y \lrar Y\) is homotopic to the identity on \(Y\). To see this, we note that by the above we have that \(fgf\colon X \lrar Y\) is in the same component as \(f\colon X \lrar Y\) in \(\cFun(X,Y)\). Using that~\eqref{e:f-map} is an equivalence of Kan complexes for \(\C=Y\) we may conclude that \(fg\colon Y \lrar Y\) is in the same component as \(\Id\colon Y \lrar Y\) in \(\cFun(Y,Y)\). Arguing as above we get that \(fg\) is homotopic to the identity with respect to the bicategorical model structure. We may hence conclude that \(f\) is an isomorphism in \(\Ho(\Ss)\) and hence a bicategorical equivalence, as desired.
\end{proof}

\section{Gray products and lax natural transformations}
\label{sec:gray-lax}

\subsection{The Gray product}\label{sec:gray-product}
In this section we define the \ndef{Gray product} of two scaled simplicial sets. 
In what follows, when we say that a \(2\)-simplex \(\sig\colon \Del^2 \to X\) \ndef{degenerates along} \(\Del^{\{i, i+1\}} \subseteq \Del^2\) (for \(i=0,1\)) we mean that \(\sig\) is degenerate and \(\sig_{|\Del^{\{i,i+1\}}}\) is degenerate. This includes the possibility that \(\sig\) factors through the surjective map \(\Del^2 \to \Del^1\) which collapses \(\Del^{\{i, i+1\}}\) as well as the possibility that \(\sig\) factors through \(\Del^2 \to \Del^0\).

\begin{define}\label{d:gray}
	Let \((X,T_X),(Y,T_Y)\) be two scaled simplicial sets. The \ndef{Gray product} \((X,T_X) \otimes (Y,T_Y)\) is the scaled simplicial set whose underlying simplicial set is the cartesian product of \(X \times Y\) and such that a \(2\)-simplex \(\sig\colon \Del^2 \to X \times Y\) is thin if and only if the following conditions hold:
	\begin{enumerate}[leftmargin=*]
		\item
		\(\sig\) belongs to \(T_X \times T_Y\); 
		\item
		either the image of \(\sig\) in \(X\) degenerates along \(\Del^{\{1,2\}}\) or the image of \(\sig\) in \(Y\) degenerates \(\Del^{\{0,1\}}\).
	\end{enumerate}
\end{define}

\begin{prop}\label{p:associative}
	The natural associativity isomorphisms of the cartesian product of marked simplicial set are also isomorphisms of scaled simplicial sets for the Gray product, and hence the Gray product gives a monoidal structure on \(\Ss\). 
\end{prop}
\begin{proof}
Let \((X,T_X),(Y,T_Y),(Z,T_Z)\) be three scaled simplicial sets. We have to check that the thin 2-simplices of \(((X,T_X)\otimes (Y,T_Y))\otimes (Z,T_Z)\) are the same as those of \((X,T_X)\otimes((Y,T_Y)\otimes (Z,T_Z))\). Indeed, direct inspection shows that both sets of thin 2-simplices coincide with the set of those \((\alpha,\beta,\gamma) \in T_X \times T_Y \times T_Z\) such that at least one of the following three conditions hold:
\begin{enumerate}
\item
both \(\alpha\) and \(\beta\) degenerate along \(\Del^{\{1,2\}}\);
\item
both \(\beta\) and \(\gamma\) degenerate along \(\Del^{\{0,1\}}\); or
\item
\(\alpha\) degenerates along \(\Del^{\{1,2\}}\) and \(\gamma\) degenerate along \(\Del^{\{0,1\}}\).
\end{enumerate}
\end{proof}

\begin{rem}\label{r:unit}
	The \(0\)-simplex \(\Del^0\) can be considered as a scaled simplicial set in a unique way, and serves as the unit of the Gray product. Consequently, if \(X\) is any \ndef{discrete} scaled simplicial set then \(X \otimes Y \cong X \times Y\) and \(Y \otimes X \cong Y \times X\) for any scaled simplicial set \(Y\).
\end{rem}

\begin{rem}\label{r:symmetric}
	The Gray product is \ndef{not symmetric} in general. Instead, there is a natural isomorphism
	\[ X \otimes Y \cong (Y^{\op} \otimes X^{\op})^{\op}. \]
\end{rem}

\begin{example}\label{ex:lax-square}
	Consider the Gray product \(X = \Del^1 \otimes \Del^1\). Then \(X\) has exactly two non-degenerate triangles \(\sig_1,\sig_2\colon \Del^2 \to X\), where \(\sig_1\) sends \(\Del^{\{0,1\}}\) to \(\Del^{\{0\}} \times \Del^1\) and \(\Del^{\{1,2\}}\) to \(\Del^1 \times \Del^{\{1\}}\), and \(\sig_2\) sends \(\Del^{\{0,1\}}\) to \(\Del^1 \times \Del^{\{0\}}\) and \(\Del^{\{1,2\}}\) to \(\Del^{\{1\}} \times \Del^1\). By definition we see that \(\sig_2\) is thin in \(X\) but \(\sig_1\) is not. If \(\C\) is an \(\infty\)-bicategory then a map \(p \colon X \to \C\) can be described as a diagram in \(\C\) of the form
	\[
	 \begin{tikzcd}[column sep=3em, row sep=large]
	  x \ar[r, "f_0"] \ar[d, "g_0"'] \ar[rd, "h"{description}, ""{swap, name=diag}] &
	  y \ar[d, "g_1"] \\
	  z \ar[r, "f_1"{swap}] & w
	  \ar[Rightarrow, from=diag, to=2-1]
	  \ar[from=2-1, to=1-2, phantom, "\simeq"{description, pos=0.75}]
	 \end{tikzcd}
	\]
	whose upper right triangle is thin (here \(f_i = p_{|\Del^1 \times \{i\}}\) and \(g_i = p_{|\{i\} \times \Del^1}\)). We thus have an invertible \(2\)-cell \(h \x{\simeq}{\Longrightarrow} g_1 \circ f_0\) and a non-invertible \(2\)-cell \(h \Rightarrow f_1 \circ g_0\). Such data is essentially equivalent to just specifying a single non-invertible \(2\)-cell \(g_1 \circ f_0 \Rightarrow f_1 \circ g_0\). We may hence consider such a square as a \ndef{oplax-commutative} square, or a square which commutes up to a prescribed \(2\)-cell.
\end{example}

It is straightforward to verify that the Gray product preserves cofibrations and colimits separately in each variable. Consequently, one my associate with \(\otimes\) a right and a left mapping objects, which we shall denote by \(\LMap(X,Y)\) and \(\RMap(X,Y)\) respectively. More explicitly, an \(n\)-simplex of \(\LMap(X,Y)\) is given by a map of scaled simplicial sets
\[ \Del^n \otimes X \longrightarrow Y .\]
A \(2\)-simplex \(\Del^2 \otimes X \to Y\) of \(\LMap(X,Y)\) is thin if it factors through \((\Del^2)^{\sharp} \otimes X\). Similarly, an \(n\)-simplex of \(\RMap(X,Y)\) is given by a map of scaled simplicial sets
\[ X \otimes \Del^n \longrightarrow Y \]
and the scaling is determined as above.

Let \(\C\) be an \(\infty\)-bicategory and \(K\) a scaled simplicial set. We will see below that the scaled simplicial sets \(\LMap(K,\C)\) and \(\RMap(K,\C)\) are in fact \(\infty\)\nbd-bi\-categories. The objects of \(\LMap(K,\C)\) correspond to functors \(K \to \C\) and by Example~\ref{ex:lax-square} we may consider morphisms in \(\LMap(K,\C)\) as \ndef{oplax natural transformations}. If we take \(\RMap(K,\C)\) instead then the objects are again functors \(K \to \C\), but now the edges will correspond to \ndef{lax natural transformations}. For example, if \(K = \Del^1\) and \(f_0,f_1\) are two morphisms in \(\C\) then an arrow in \(\LMap(\Del^1,\C)\) from \(f_0\) to \(f_1\) is a square in \(\C\) of the form
\begin{equation}\label{e:square}
\begin{tikzcd}[column sep=3em, row sep=large]
x \ar[r, "f_0"] \ar[d, "g_0"'] \ar[rd, "h"{description}, ""{swap, name=diag}] &
y \ar[d, "g_1"] \\
z \ar[r, "f_1"{swap}] & w
\ar[Rightarrow, from=diag, to=2-1]
\ar[from=2-1, to=1-2, phantom, "\simeq"{description, pos=0.75}]
\end{tikzcd}
\end{equation}
with the upper right triangle thin. This can be considered as a \(2\)-cell \(g_1 \circ f_0 \Rightarrow f_1 \circ g_0\). On the other hand, an arrow in \(\RMap(\Del,\C)\) from \(f_0\) to \(f_1\) is a square of the form~\ref{e:square} whose \ndef{lower left} triangle is thin, i.e., a \(2\)-cell in the other direction \(f_1 \circ g_0 \Rightarrow g_1 \circ f_0\).

\begin{rem}\label{r:assoc}
	Let \(X\), \(Y\) and \(Z\) be scaled simplicial sets. The associativity isomorphism of the Gray product yields a natural isomorphism
	\[ \LMap(X,\RMap(Y,Z)) \cong \RMap(Y,\LMap(X,Z)) .\]
\end{rem}

\subsection{Comparison with the Gray product of stratified sets}\label{sec:comparison}

In his extensive work~\cite{VerityWeakComplicialI}, Verity constructs a Gray tensor product in the setting of stratified sets. In particular, for two stratified sets \((X,t_X),(Y,t_Y)\), the stratified set \((X,t_X) \otimes (Y,t_Y)\) has as an underlying marked simplicial set the product of the underlying marked simplicial sets, while a 2-simplex \((\sigma_X,\sigma_Y)\colon \Delta^2\to X\times Y\) is marked in \((X,t_X) \otimes (Y,t_Y)\) if and only if
\begin{enumerate}
\item \(\sigma_X\) and \(\sigma_Y\) are marked in \(X\) and \(Y\) respectively;
\item either \((\sigma_X)_{|\Delta^{\{1,2\}}}\) is marked in \(X\) or \((\sigma_Y)_{|\Delta^{\{0,1\}}}\) is marked in \(Y\).
\end{enumerate}
The marking on higher simplices are defined in a similar manner, though they do not play a significant role if one is only considering stratified sets up to 2-complicial weak equivalence.
Our goal in the present subsection is to compare the Gray product defined in \S\ref{sec:gray-product} with the above Gray tensor product of stratified sets via the left Quillen equivalence
\[ \Ss \to \st_2 \]
established in~\cite{Equivalence}.

\begin{rem}
The Gray product of stratified sets recalled above is associative, but does not preserve colimits in each variable, and in particular cannot be a left Quillen bifunctor. In~\cite{VerityWeakComplicialI}, Verity also considers a variant of the above definition which preserve colimits in each variable but is not associative. By contrast, the Gray tensor product of scaled simplicial sets is simultaneously associative and a left Quillen bifunctor, as we will establish in Theorem~\ref{t:gray-left-quillen} below.
\end{rem}

In what follows, it will be useful to consider several equivalent variants of the Gray tensor product on scaled simplicial sets. 
Let \((X,T_X),(Y,T_Y)\) be scaled simplicial sets and let \(T_{\gr} \subseteq T_X \times T_Y\) denote the collection of triangles which are thin in the Gray product \((X,T_X) \otimes (Y,T_Y)\), see Definition~\ref{d:gray}. Let \(T_- \subseteq T_X \times T_Y\) denote the subset consisting of those pairs of thin triangles \((\sig_X,\sig_Y)\)
for which either both \(\sig_X\) and \(\sig_Y\) are degenerate or at least one of \(\sig_X,\sig_Y\) degenerates to a \ndef{point}. On the other hand, let \(T_+ \subseteq T_X \times T_Y\) be the set of those pairs of thin triangles \((\sig_X,\sig_Y)\) 
for which 
either \((\sig_X)_{|\Del^{\{1,2\}}}\) is degenerate or \((\sig_Y)_{|\Del^{\{0,1\}}}\) is degenerate. Then we have a sequence of inclusions
\[ T_- \subseteq T_{\gr} \subseteq T_+.\]
We claim that these three choices for the collection of thin triangles in \(X \times Y\) yield equivalent models for the Gray tensor product. More precisely, we have the following: 

\begin{prop}\label{p:variants}
	The maps \((X \times Y,T_-) \hrar (X \times Y,T_{\gr}) \hrar (X \times Y,T_+)\) are both scaled anodyne.
\end{prop}

The proof of Proposition~\ref{p:variants} will require a couple of lemmas.

\begin{lemma}\label{l:lem-1}
	In the situation of Proposition~\ref{p:variants}, if \((X,T_X)=\Del^2_{\flat}\) and \((Y,T_Y) = \Del^1_{\flat}\) then the map \((X \times Y,T_-) \to (X \times Y,T_{\gr})\) is scaled anodyne.
\end{lemma}
\begin{proof}
	We note that in this case \(T_{\gr}\) contains exactly one triangle that is not in \(T_-\), namely, the triangle \(\sig \colon \Del^2 \to \Del^2 \times \Del^1\) whose projection to \(\Del^2\) is the identity and whose projection to \(\Del^1\) is surjective and degenerates along \(\Del^{\{0,1\}}\). Let \(\Del^3 \to \Del^2 \times \Del^1\) be the \(3\)-simplex spanned by the vertices \((0,0),(1,0),(2,0),(2,1)\). By definition we have that \(\sig_{|\Del^{\{1,2,3\}}},\sig_{|\Del^{\{0,2,3\}}}\) and \(\sig_{|\Del^{\{0,1,2\}}}\) are in \(T_-\), while \(\sig_{|\Del^{\{0,1,3\}}}\) is exactly the \(2\)-simplex which is in \(T_{\gr}\) but not in \(T_-\). We then get that the map \((X \times Y,T_-) \to (X \times Y,T_{\gr})\) is a pushout along the inclusion 
	\[(\Del^3,\{\Del^{\{1,2,3\}},\Del^{\{0,2,3\}},\Del^{\{0,1,2\}}\}) \hrar \Del^3_{\sharp}\] 
	which is scaled anodyne by~\cite[Remark 3.1.4]{LurieGoodwillie}.
\end{proof}

\begin{lemma}\label{l:lem-2}
	Let 
	\[ (Z_0,T_0) = \biggl(\Del^2 \coprod_{\Del^{\{1,2\}}}\Del^0\biggr) \otimes \Del^2 \] 
	and 
	\[ (Z_2,T_2) = \Del^2 \otimes \biggl(\Del^2 \coprod_{\Del^{\{0,1\}}}\Del^0\biggr) .\] 
	For \(i=0,2\), let \(S_i\) be the set of \(2\)-simplices consisting of \(T_i\) together with the image of the diagonal \(2\)-simplex \(\diag \colon \Del^2 \to \Del^2 \times \Del^2\). Then the map of scaled simplicial sets
	\[ (Z_i,T_i) \longrightarrow (Z_i,S_i) \]
	is scaled anodyne for \(i=0,2\).
\end{lemma}
\begin{proof}
	We prove the claim for \(i=0\). The proof for \(i=2\) proceeds in a similar manner. 
	Let \(\sig \colon\Del^3 \to \Del^2 \times \Del^2\) be the \(3\)-simplex spanned by the vertices \((0,0),(1,0),(1,1),(2,2)\), and let \(\sig' \colon \Del^3 \to Z_0\) be its image in \(Z_0\). 
	By definition we have that \(\sig'_{|\Del^{\{1,2,3\}}},\sig'_{|\Del^{\{0,1,3\}}}\) and \(\sig'_{|\Del^{\{0,1,2\}}}\) are thin in \(Z_0\), and \(\sig'_{|\Del^{\{0,2,3\}}}\) is the image of the diagonal \(2\)-simplex. We then get that the map \((Z_0,T_0) \to (Z_0,S_0)\) is a pushout along the inclusion 
	\[(\Del^3,\{\Del^{\{1,2,3\}},\Del^{\{0,1,3\}},\Del^{\{0,1,2\}}\}) \hrar \Del^3_{\sharp}\] 
	which is scaled anodyne by~\cite[Remark 3.1.4]{LurieGoodwillie}.	
\end{proof}

\begin{proof}[Proof of Proposition~\ref{p:variants}]
	The inclusion \((X \times Y,T_-) \hrar (X \times Y,T)\) can be obtained as a sequence of pushouts along the scaled anodyne maps described in Lemma~\ref{l:lem-1}, while the inclusion \((X \times Y,T) \hrar (X \times Y,T_+)\) can be obtained as a sequence of pushouts along the scaled anodyne maps described in Lemma~\ref{l:lem-2}. 
\end{proof}

\begin{cor}[Comparison of scaled and 2-complicial Gray products]
\label{c:comparison-scaled-stratified}
For scaled simplicial sets \((X,T_X),(Y,T_Y)\), the natural inclusion \(i_{X,Y}\colon\iota ((X,T_X)\otimes (Y,T_Y))\to \iota(X,T_X) \otimes \iota (Y,T_Y)\) is a weak equivalence in \(\st_2\).
\end{cor}
\begin{proof}
The map \(i_{X,Y}\) coincides with \(\iota((X\times Y,T_{\gr}) \hookrightarrow (X\times Y,T_+))\). Since \(\iota\) is a left Quillen functor and \((X\times Y,T_{\gr}) \hookrightarrow (X\times Y,T_+)\) is scaled anodyne by Proposition \ref{p:variants} this map is a weak equivalence in the 2-complicial model structure.
\end{proof}

We finish this section with some additional results concerning the relation between the Gray product of scaled simplicial sets and invertible arrows in \(\infty\)-bicategories (see Definition~\ref{d:equivalence}).

\begin{prop}\label{p:invertible-leg}
Let \(\C = (\ovl{\C},T_{\C}),\D=(\ovl{\D},T_{\D})\) be two \(\infty\)-bicategories. Let \(T_{\minisimeq} \subseteq T_{\C} \times T_{\D}\) denote the subset containing those triangles \((\alp,\beta)\) such that either \(\alp|_{\Del^{\{1,2\}}}\) is invertible in \(\C\) or \(\beta|_{\Del^{\{0,1\}}}\) is invertible in \(\D\). Then the map
\[ \C \otimes \D \to (\ovl{\C} \times \ovl{\D},T_{\minisimeq})\]
is bicategorical equivalence.
\end{prop}
\begin{proof}
Let \(T_{\gr} \subseteq T_{\C} \times T_{\D}\) be the collection of triangles which are thin in \(\C \otimes \D\). By Proposition~\ref{p:variants} it will suffice to show that \(T_{\minisimeq}\) is contained in the saturated closure of \(T_+\). For this, let \((\alp,\beta) \in T_{\minisimeq}\) be a triangle, so that either \(\alp|_{\Del^{\{1,2\}}}\) is invertible in \(\C\) or \(\beta|_{\Del^{\{0,1\}}}\) is invertible in \(\D\). Assume first that \(\beta|_{\Del^{\{0,1\}}}\) is invertible. 
Since \(\D\) is an \(\infty\)-bicategory we have that \(\D^{\thi}\) is an \(\infty\)-category (which contains the triangle \(\beta\)) and hence we may find a map 
\(\rho \colon \Del^4 \to \D^{\thi}\) 
such that \(\rho_{|\Del^{\{0,1,4\}}} = \bet\), \(\rho_{|\Del^{\{0,2\}}}\) is degenerate on \(\beta(0)\) and \(\rho_{|\Del^{\{1,3\}}}\) is degenerate on \(\beta(1)\). Let \(\eta\colon \Del^4 \to \C\) be the composed map \(\eta \colon\Del^4 \x{\pi}{\to} \Del^2 \x{\alp}{\to} \C\) where \(\pi \colon \Del^4 \to \Del^2\) is the map which is given on vertices by \(\pi(0)=0\), \(\pi(1)=\pi(2)=\pi(3)=1\) and \(\pi(4)=2\). Let 
\begin{equation}\label{e:saturation-anodyne}
	 (\Del^4,T) \longrightarrow (\Del^4,T \cup \{\Del^{\{0,1,4\}},\Del^{\{0,3,4\}}),
\end{equation}
be the scaled anodyne map of Definition~\ref{d:anodyne}~(ii).
	We now claim that the map \((\eta,\rho)\colon \Del^4 \to \ovl{\C} \times \ovl{\D}\) sends \(T\) to \(T_+\). Indeed, if \(\Del^{\{i,j,k\}} \in T\) then either \((i,j)\in \{(0,2),(1,3)\}\), in which case \(\rho(\Del^{\{i,j\}})\) is degenerate, or \(\{j,k\} \subseteq \{1,2,3\}\), in which case \(\eta(\Del^{\{j,k\}})\) is degenerate. It then follows in particular 
	\((\alp,\bet) = (\eta,\rho)(\Del^{\{0,1,4\}})\) 
	is contained in the saturated closure of \(T_+\). 
	Finally, if we assume that it is \(\alp|_{\Del^{\{1,2\}}}\) that is invertible instead of \(\beta|_{\Del^{\{0,1\}}}\), then the argument can be carried out in a symmetric manner except that we need to use the opposite of the map~\eqref{e:saturation-anodyne}. \end{proof}

\begin{cor}\label{c:kan-bicat-gray}
	Let \(\C\) be an \(\infty\)-bicategory and \(K\) a Kan complex. Then the maps 
\[\C \otimes K_{\sharp} \longrightarrow \C \times K_{\sharp}\]
	and 
\[K_{\sharp} \otimes \C \longrightarrow K_{\sharp} \times \C\]
	are scaled anodyne.
\end{cor}

\subsection{The Gray product as a left Quillen bifunctor}\label{sec:gray-is-quillen}

In this section we will prove the main result of the present paper:
\begin{thm}\label{t:gray-left-quillen}
	The Gray tensor product is a left Quillen bifunctor
	\[
	 -\otimes -\colon \Ss \times \Ss \longrightarrow \Ss
	\]
	with respect to the bicategorical model structure.
\end{thm}

Combined with Proposition~\ref{p:associative} this implies that \((\Ss,\otimes)\) is a monoidal model category. Passing to underlying \(\infty\)-categories we conclude:

\begin{cor}
The Gray product endows the \(\infty\)-category \(\Cat_{(\infty,2)}\) with a monoidal structure which is compatible with colimits in each variable.
\end{cor}

We will give the proof of Theorem~\ref{t:gray-left-quillen} below. The following pushout-product property constitutes the principal component of the proof:
\begin{prop}\label{p:push-prod}
Let \(f\colon X \lrar Y\) be a monomorphism of scaled simplicial sets and
\(g\colon Z \lrar W\) be a scaled anodyne map. 
Then the pushout-products
\[
 f \hgr g \colon 
 \bigl( X \otimes W \bigr) \coprod_{X \otimes Z}
 \bigl( Y \otimes Z \bigr) \longrightarrow Y \otimes W
\]
and
\[
 g \hgr f \colon 
 \bigl( W \otimes X \bigr) \coprod_{Z \otimes X}
 \bigl( Z \otimes Y \bigr) \longrightarrow W \otimes Y
\]
are scaled anodyne maps of scaled simplicial sets.
\end{prop}

\begin{proof}
We adapt the argument of~\cite[Proposition 3.1.8]{LurieGoodwillie} to the context of Gray products. We can assume that \(f\) is either the inclusion \(\partial \Del^n_{\flat} \hrar \Del^n_{\flat}\) for \(n \geq 0\) or the inclusion or the inclusion \(\Del^2_{\flat} \subseteq \Del^2_\sharp\), and that \(g\) is one of the generating scaled anodyne maps appearing in Definition~\ref{d:anodyne}. The case where \(f\) is the inclusion \(\partial \Del^0 \hrar \Del^0\) is trivial since in this case both \(f \hgr g\) and \(g \hgr f\) are isomorphic to \(g\), see Remark~\ref{r:unit}. If \(f\) is the inclusion \(\Del^2_{\flat} \subseteq \Del^2_\sharp\) then, since all generating anodyne maps in Definition~\ref{d:anodyne} are bijective on vertices, the map \(f \hgr g\) becomes an isomorphism if we replace the collection of thin triangles in the Gray product by its minimalist variant \(T_{-}\) as in \S\ref{sec:comparison}. The statement that \(f \hgr g\) is scaled anodyne can then be deduced from Proposition~\ref{p:variants} (and the same argument works for \(g \hgr f\)). 
We may hence assume that \(f\) is the map \(\partial \Del^n_{\flat} \hrar \Del^n_{\flat}\) for \(n \geq 1\). We now need to address the three different possibilities for \(g\) appearing in Definition~\ref{d:anodyne}.

\begin{enumerate}[leftmargin=*]
	\item[(A)]
	Suppose that \(g\) is the inclusion \(\left(\Lam^m_i,T'\right)  \hrar \left(\Del^m,T\right)\) for \(0 < i < m\), where \(T\) denotes the union of all degenerate edges and \(\{\Del^{\{i-1,i,i+1\}}\}\) and \(T' = T_{|(\Lam^m_i)}\). We argue the case of \(g \hgr f\). The proof for \(f \hgr g\) proceeds in a similar manner. 
 Let
\[ (Z_0,M_0) = \left[(\Del^m,T) \otimes \partial\Del^n_{\flat}\right] \coprod_{(\Lam^m_i,T') \otimes \partial\Del^n_{\flat}} \left[(\Lam^m_i,T') \otimes \Del^n_{\flat}\right] \]
We will extend \((Z_0,M_0)\) to a filtration of \((\Del^m,T) \otimes \Del^n_{\flat}\) as follows.
Let \(S\) denote the collection of all simplices \(\sig\colon \Del^{k_\sig} \to \Del^m \times \Del^n\) with the following properties:
\begin{enumerate}
\item[(i)]
the simplex \(\sigma\) is non-degenerate and induces surjections \(\Delta^{k_{\sigma}} \to \Delta^n\) and \(\Del^{k_{\sigma}} \to \Del^m\) along the projections;
\item[(ii)] 
there exist integers \(0 < p_{\sigma} < k_{\sigma}\) and \(0 < j_\sigma \leq m\) (necessarily unique) such that \(\sigma(p_{\sigma} - 1) = (i-1, j_\sigma)\) and \(\sigma(p_{\sigma}) = (i, j_\sigma)\).
\end{enumerate}
We make the following observation, which is useful to keep in mind during the arguments below: if \(\tau\colon \Del^k \to \Del^m \times \Del^n\) is an arbitrary simplex, then \(\tau\) belongs to \(Z_0\) unless its projection to \(\Del^n\) is surjective and its projection to \(\Del^m\) contains the face opposite \(i\). In the latter case, if the image of \(\tau\) in \(\Del^m\) is exactly the face opposite \(i\) then there is a unique \((k+1)\)-simplex \(\sig \in S\) of which \(\tau\) is a face: indeed, if \(0 \leq p \leq k+1\) is the maximal number such that \(\tau(p-1) = (i-1,j)\) for some \(j \in [m]\), then the only \((k+1)\)-simplex in \(S\) which has \(\tau\) as a face must send \(p\) to \((i,j)\) and have \(\sig\) as its face opposite \(p\). Otherwise, if the projection of \(\tau\) to \(\Del^m\) is surjective and \(p\) is defined in the same manner then either \(\tau\) itself belongs to \(S\) 
or 
\(\tau\) is the face of exactly two simplices \(\sig,\sig'\) which belong to \(S\), one with \(\sig(p)=(i,j)\) so that \(p_{\sig}=p\) and \(j_{\sig}=j\) and one with \(\sig'(p)=(i-1,j+1)\) so that \(p_{\sig'}=p+1\) and \(j_{\sig}=j+1\). In particular, \(\Del^m \times \Del^n\) is obtained from \(Z_0\) by adding all the simplices in \(S\). To proceed, we will need to identify the right order in which to add them.

Choose an ordering \(\sig_1 < ... < \sig_{\ell}\) on \(S\) such that \(a < b\) whenever \(\dim(\sig_a) < \dim(\sig_b)\) or \(\dim(\sig_a) = \dim(\sig_b)\) and \(j_{\sig_a} < j_{\sig_b}\). We then abbreviate \(k_a := k_{\sig_a}, p_a = p_{\sig_a}\) and \(j_a := j_{\sig_a}\). We now observe that for \(a=1,...,\ell\) we have \(p_a < k_a\), otherwise the projection of \(\sig_a\) to the \(\Del^m\) coordinate will not be surjective (since it sends \(p_a\) to \(i<n\)). We then denote by \(T_a \subseteq \Del^{k_a}_2\) the union of all degenerate \(2\)-simplices and the 2-simplex \(\Del^{\{p_a-1,p_a,p_a+1\}}\), and set \(T_a' = T_a \cap \left(\Lam^{k_a}_{p_a}\right)_2\). We now claim that each \(\sig_a \in S\) maps \(T_a\) into the set of thin triangles of \((\Del^m,T) \otimes \Del^n_{\flat}\). Indeed, it suffice to observe that \(\sig_a\) sends the \(2\)-simplex \(\Del^{\{p_a-1,p_a,p_a+1\}}\) to either a degenerate simplex in \(\Del^m\) or to \(\Del^{\{i-1,i,i+1\}}\), and on the other hand sends the same triangle to a 2-simplex of \(\Del^n_{\flat}\) which degenerates along \(\Del^{\{p_a,p_a+1\}}\). In particular, we may view each \(\sig_a\) as a map of scaled simplicial sets
\[ \sig_a\colon (\Del^{k_a},T_a) \to (\Del^m,T) \otimes \Del^n_{\flat}.\]
Now for \(a=1,...,\ell\) let \(Z_a \subseteq \Del^m \times \Del^n\) to be the union of \(Z_0\) and the images of the simplices \(\sig_{a'}\) for \(a' \leq a\), and let \(M_a\) be the union of the images \(\sig_{a'}(T_{a'})\) for \(a' \leq a\). We claim that for \(a=1,...,\ell\) we have a pushout square of the form
\[ \xymatrix{
\left(\Lam^{k_a}_{p_a},T_a'\right) \ar[r]\ar[d] & (Z_{a-1},M_{a-1}) \ar[d] \\
\left(\Del^{k_a},T_a\right) \ar[r]^-{\sig_a} & (Z_a,M_a) \\
}\]
To prove this, it will suffice to show that all the faces of \(\sig_a\) are contained in \(Z_{a-1}\). Indeed, for \(k' \neq p_a-1,p_a\) the restriction of \(\sig_a\) to the face opposite \(k'\) is either contained in \(Z_0\) or is a \((k_a-1)\)-simplex in \(S\). In the latter case it will correspond to an index \(a' < a\) and will be contained in \(Z_{a'} \subseteq Z_{a-1}\). Now suppose that \(k'=p_a-1\). In this case either \(\sig_a\) sends the face opposite \(k'\) to \(Z_0\) or \(p_a \geq 2\) and there exists some \(j' < j_a\) such that \(\sig_a(p_a-2) = (i-1,j')\). In the latter case the face opposite \(k'\) is also the face of a simplex \(\sig_{a'}\colon \Del^{k_{a'}} \to \Del^m \times \Del^n\) in \(S\) with \(k_{a'} = k_a\),\(p_{a'} = p_a-1\) and \(j_{a'}=j' < j_a\). Then \(a' < a\) and this face will again belong \(Z_{a-1}\). Finally, if \(k'=p_a\) then we claim that the image under \(\sig_a\) of the face opposite \(l\) will not belong to \(Z_{a-1}\). To see this, we note that \(\sig_a(p_a+1)\) must be either \((i,j+1),(i+1,j)\) or \((i+1,j+1)\). We then see that in all three cases the the restriction of \(\sig_a\) to the face opposite \(p_a\) does not belong to \(Z_0\) and does not belong to \(S\). In addition, by the observation following the definition of \(S\) above, when \(\sig_a(p_a+1)=(i+1,j),(i+1,j+1)\) the \((k_a-1)\)-simplex in question is the face of a unique simplex in \(S\), namely, \(\sig_a\), and when \(\sig_a(p_a+1) = (i,j+1)\) it is the face of exactly two simplices in \(S\), one of which is \(\sig_a\) and the other is \(\sig_{a'}\) for which \(j_{a'} = j+1\) and so \(a' > a\). We may thus conclude that the the restriction of \(\sig_a\) to the face opposite \(p_a\) does not belong to \(Z_a\). We may hence conclude that the inclusion \((Z_0,M_0) \hrar (Z_{\ell},M_{\ell})\) is scaled anodyne.

Now when \(m \geq 3\) every thin \(2\)-simplex of \((\Del^m,T)\) is contained in \((\Lam^m_i,T)\), and when \(n \geq 2\) every thin 2-simplex of \(\Del^n\) is contained in \(\partial \Del^n\). In either of these cases we have that every thin triangle in \((\Del^m,T) \otimes \Del^n_{\flat}\) is contained in \(Z_0\) and so \((Z_{\ell},M_{\ell})=(\Del^m,T) \otimes \Del^n_{\flat}\) as scaled simplicial sets and the proof is complete. 
In the special case \(n=1\) and \(m=2\) we have
\[
S = \{\sigma_1 < \tau_2 < \tau_1 < \tau_0\},
\]
where \(\sigma_1\) is the triangle of \(\Del^2 \times \Del^1\) with vertices
\[
\big((0,0), (1,0), (2,1)\big)
\]
and \(\tau_j\), \(j=0, 1, 2\) are the \(3\)-simplices given by the map
\[
 \tau_j(l) = \begin{cases} (0,l) & l \leq j , \\ (1,l-1) & l > j .\end{cases} \]
We then see that \(M_{\ell}\) contains all the triangles which are thin in \(\Del^2_{\sharp} \times \Del^1_{\flat}\) except the one with vertices \((0,0),(2,0),(2,1)\). 
To finish the proof of this case we then consider the pushout square 
\[ \xymatrix{
(\Del^3,\{\Del^{\{0,1,2\}},\Del^{\{0,1,3\}},\Del^{\{1,2,3\}}\}) \ar[r]\ar[d] & (Z_{\ell},M_{\ell}) \ar[d] \\
\Del^3_{\sharp} \ar[r] & \Del^2_{\sharp} \otimes \Del^1 \\
}\]
where 
the map \(\Del^3_{\sharp} \lrar \Del^2_{\sharp} \otimes \Del^1\) is the one determined by the chain of vertices \((0,0),(1,0),(2,0),(2,1)\).  
By~\cite[Remark 3.1.4]{LurieGoodwillie} we get that the map \((Z_{\ell},M_{\ell}) \to \Del^2_{\sharp} \otimes \Del^1\) is scaled anodyne.

\item[(B)]
The case where \(g\) is the inclusion
\( \left(\Del^4,T\right) \hrar \left(\Del^4, T \cup \{\Del^{\{0,3,4\}},\Del^{\{0,1,4\}}\}\right) \)
where \(T\) is the set of triangles specified in Definition~\ref{d:anodyne}(B). 
If \(n \geq 2\) then both \(f \hgr g\) and \(g \hgr f\) are isomorphisms of scaled simplicial sets and so we may assume that \(n=1\).

Let us denote the domain of \(f \hgr g\) by \((\Del^1 \times \Del^4,T')\) and the domain of \(g \hgr f\) by \((\Del^4 \times \Del^1,T'')\). Let \(p \colon \Del^4 \to \Del^1\) be the unique map which sends \(0\) to \(0\) and \(1,2,3\) to \(1\) and let \(q\colon \Del^4 \to \Del^1\) be the unique map which sends \(0,1,2\) to \(0\) and \(3\) to \(1\). 
To show that \(f \hgr g\) and \(g \hgr f\) are scaled anodyne we then observe that there are pushout diagrams of the form
\[
\begin{tikzcd}
\left(\Del^4,T\right) \ar[r, "{p \times \Id}"]\ar[d] & (\Del^1 \times \Del^4,T')\ar[d, "{f \hgr g}"] \\
\left(\Del^4, T \cup \{\Del^{\{0,3,4\}},\Del^{\{0,1,4\}}\}\right) \ar[r]   & \Del^1 \otimes \left(\Del^4, T \cup \{\Del^{\{0,3,4\}},\Del^{\{0,1,4\}}\}\right)
\end{tikzcd}
\]
and
\[
\begin{tikzcd}
\left(\Del^4,T\right) \ar[r, "{\Id \times q}"]\ar[d] & (\Del^4 \times \Del^1,T'') \ar[d, "{g \hgr f}"]   \\
\left(\Del^4, T \cup \{\Del^{\{0,3,4\}},\Del^{\{0,1,4\}}\}\right) \ar[r] & \left(\Del^4, T \cup \{\Del^{\{0,3,4\}},\Del^{\{0,1,4\}}\}\right) \otimes \Del^1
	 \end{tikzcd}.
	\]

\item[(C)]
The case where \(g\) is in the inclusion \( \left(\Lam^m_0 \coprod_{\Del^{\{0,1\}}}\Del^0,T\right)  \hrar \left(\Del^m \coprod_{\Del^{\{0,1\}}}\Del^0,T\right) \)
for \(m \geq 3\), where \(T\) denotes the union of all degenerate edges together with the triangle \(\Del^{\{0,1,n\}}\). We prove the case of \(g \hgr f\). The proof for \(f \hgr g\) proceeds in a similar manner. 
		
We argue as in the proof of case (A). Let \(S\) denote the collection of all simplices \(\sig\colon \Del^{k_{\sig}} \to \Del^m \times \Del^n\) with the following properties:
\begin{enumerate}
\item[(i)]
the simplex \(\sig\) is non-degenerate, and induces surjections \(\Del^{k_\sig} \lrar \Del^m\) and \(\Del^{k_\sig} \to \Del^n\);
\item[(ii)]
there exist integers \(0 \leq p_{\sigma} < k_{\sigma}\) and \(0 < j_\sigma \leq m\) (necessarily unique) such that \(\sigma(p_{\sigma}) = (0, j_\sigma)\) and \(\sigma(p_{\sigma}+1) = (1, j_\sigma)\).
\end{enumerate}	
We may then choose an ordering \(\sig_1 < ... < \sig_q\) on \(S\) such that \(a < b\) whenever \(\dim(\sig_a) < \dim(\sig_b)\) or \(\dim(\sig_a) = \dim(\sig_b)\) and \(j_{\sig_a} > j_{\sig_b}\). We then abbreviate \(k_a := k_{\sig_a}, p_a = p_{\sig_a}\) and \(j_a := j_{\sig_a}\). For every index \(a\), let \(T_a \subseteq \Del^{k_a}_2\) denote the collection of all \(2\)-simplices which are either degenerate, have the form \(\Del^{\{p_a-1,p_a,p_a+1\}}\) if \(p_a > 0\), or have the form \(\Del^{\{0,1,m\}}\) if \(p_a = 0\). Let \(T_a' \subseteq T_a\) be the subset of those \(2\)-simplices in \(T_a\) which lie in \(\Lam^{k_a}_{p_a}\). We claim that each \(\sig_a \in S\) maps \(T_a\) into the set of thin triangles in \((\Del^m \coprod_{\Del^{\{0,1\}}}\Del^0,T) \otimes \Del^n_{\flat}\).
	To see this it suffice to observe that \(\sig_a\) always sends the \(1\)-simplex \(\Del^{\{p_a, p_a+1\}}\) to a degenerate \(1\)-simplex in both \(\Del^m \coprod_{\Del^{\{0,1\}}}\Del^0\) and \(\Del^n\).
	
	Now define a sequence of scaled simplicial sets
	\[
	 (Z_0,M_0) \subseteq (Z_1,M_1) \subseteq \dots \subseteq (Z_{\ell},M_{\ell}) \subseteq
	 \biggl(\Del^m \coprod_{\Del^{\{0,1\}}}\Del^0,T\biggr) \otimes\Del^n_{\flat}
	\]
	where \((Z_0,M_{\ell})\) is the domain of \(g \hgr f\) and for every \(a=1,...,\ell\) we let \(Z_a\) be the union of \(Z_0\) with the images of \(\sig_{a'}\) for \(a' \leq a\) and \(M_a\) the union of \(M_0\) with the images \(\sig_{a'}(T_{a'})\) for all \(a' \leq a\).
Arguing as in the case (A) we now observe that 
\((Z_{\ell},M_{\ell}) = \biggl(\Del^m \coprod_{\Del^{\{0,1\}}}\Del^0,T\biggr) \otimes\Del^n_{\flat}\) (note that \(m \geq 3\))
and for \(a=1,...,\ell\) we have a pushout diagram
\[\xymatrix{
\left(\Lam^{k_a}_{p_a},T_a'\right) \ar[r]\ar[d] & (Z_{a-1},M_{a-1}) \ar[d] \\
\left(\Del^{k_a},T_a\right) \ar[r]^{\sig_a} & (Z_a,M_a) \\
}\]
where in the case that \(p_a=0\) the simplex \(\sig_a\) sends \(\Del^{\{0,1\}}\) to a degenerate edge of \(Z_{a-1}\). We may then conclude that each \((Z_{a-1},M_{a-1}) \to (Z_a,M_a)\) is scaled anodyne and so the desired result follows.

\qedhere
\end{enumerate}
\end{proof}

\begin{cor}\label{c:kan-gray}
	Let \((X,T_X)\) be a scaled simplicial set and \(K\) a Kan complex. Then the maps 
	\begin{equation}\label{e:kan-gray-1}
	(X,T_X) \otimes K_{\sharp} \longrightarrow (X,T_X) \times K_{\sharp}
	\end{equation}
	and 
	\begin{equation}\label{e:kan-gray-2}
	K_{\sharp} \otimes (X,T_X) \longrightarrow K_{\sharp} \times (X,T_X)
	\end{equation}
	are trivial cofibrations.
\end{cor}
\begin{proof}
We prove that~\eqref{e:kan-gray-1} is a bicategorical equivalence, the proof for~\eqref{e:kan-gray-2} proceeds in a similar manner.
Let \(\C\) be an \(\infty\)-bicategory equipped with a scaled anodyne map \((X,T_X) \hrar \C\). We then obtain a commutative square
\[ \xymatrix{
(X,T_X) \otimes K_{\sharp} \ar[r]\ar[d] & (X,T_X) \times K_{\sharp} \ar[d] \\
\C \otimes K_{\sharp} \ar[r] & \C \times K_{\sharp} \\
}\]
in which the vertival maps are bicategorical equivalences by Proposition~\ref{p:push-prod} and~\cite[Proposition 3.1.8]{LurieGoodwillie} and the bottom horizontal map is a bicategorical equivalence by Corollary~\ref{c:kan-gray}. It then follows that the top horizontal map is a bicategorical equivalence as well.
\end{proof}

\begin{proof}[Proof of Theorem~\ref{t:gray-left-quillen}]
	Thanks to Proposition \ref{p:push-prod} and the description of the bicategorical model structure provided in \cite{Equivalence}, we are left with proving that the maps \[(\partial \Delta^n \hookrightarrow	\Delta^n)\hat{\otimes}(\{\epsilon\}\to J_{\sharp}) \ \text{for} \ n\geq 0 \ \text{and} \ \epsilon = 0,1\] together with \[( \Delta^2 \hookrightarrow	\Delta^2_{\sharp})\hat{\otimes}(\{\epsilon\}\to J_{\sharp})\] are weak bicategorical equivalences. But by Corollary~\ref{c:kan-gray} these maps are equivalent to the corresponding ones having \(\hat{\times}\) in place of the Gray tensor product, since both \(\Delta^0\) and \(J_{\sharp}\) are maximally marked Kan complexes. Therefore, the result follows from the the fact that the bicategorical model structure on scaled simplicial sets is cartesian closed, as previously observed.
\end{proof}

\begin{cor}\label{c:map}
	Let \(\C\) be an \(\infty\)-bicategory and \(K\) a scaled simplicial set. Then \(\LMap(K,\C)\) and \(\RMap(K,\C)\) are \(\infty\)-bicategories.
\end{cor}

\begin{rem}
	Given a pair of vertices \(x,y\) in an \(\infty\)-bicategory \(\C\), the fiber of the projection \(\RMap(\Delta^1,\C)\to \C \times \C\) induced by \(\partial \Delta^1\to \Delta^1\) over \((x,y)\) is naturally isomorphic to the underlying simplicial set of \(\Hom_{\C}(x,y)\) (as defined in Notation \ref{n:mapping}).
\end{rem}

\begin{rem}\label{r:associate}
	The Gray tensor product and the cartesian product do not ``associate'', that is
	\[ K \times (X \otimes Y) \not\cong (K \times X) \otimes Y .\]
	for general scaled simplicial sets \(K,X,Y\). For this reason we also have
	\[ \Map(X \otimes Y,Z) \not\cong \Map(X,\LMap(Y,Z)) \] 
	in general (and similarly for \(\RMap(-,-)\)). However, by Corollary~\ref{c:kan-gray} and Proposition~\ref{p:push-prod} we have that if \(K\) is a Kan complex then the maps 
	\[ K_{\sharp} \times (X \otimes Y) \llar K_{\sharp} \otimes X \otimes Y \longrightarrow (K_{\sharp} \times X) \otimes Y ,\]
	are scaled anodyne (and isomorphisms on the level of the underlying simplicial sets). This means that for every \(\infty\)-bicategory \(\C\) we have an isomorphism of Kan complexes 
	\[ \cFun(X\otimes Y,\C)\cong \cFun(X,\LMap(Y,\C)).\] 
\end{rem}

\section{Oplax functors and the universal property of Gray products}
\label{sec:oplax}

The goal of this section is to characterize the Gray tensor product defined in~\S\ref{sec:gray-lax} by means of a universal property in the \(\infty\)-category of \((\infty,2)\)-categories, along the lines of what is done in Chapter 10 in the Appendix of~\cite{GaitsgoryRozenblyumStudy}. Following the approach of loc.\ cit.\ we will introduce a the notion of an oplax functor of \(\infty\)-bicategories. We will then show that for two \(\infty\)-bicategories \(\C\) and \(\D\), maps from the Gray product \(\C \otimes \D\) to any other \(\infty\)-bicategory \(\E\) can be identified as a suitable subspace of the space of oplax functor \(\C \times \D \to \E\). This opens the door to comparing the Gray product defined here with that of~\cite{GaitsgoryRozenblyumStudy} by comparing the two notions of oplax maps.

\subsection{Normalised oplax functors of 2-categories}

Before we begin, let us recall the classical notion of
\ndef{normalised oplax $2$\nbd-func\-tor} for $2$-categories.
We shall denote the horizontal composition of $1$-cells and $2$-cells
by \(\ast_0\) and the vertical composition of \(2\)-cells by~\(\ast_1\).

	Let $A$ and $B$ be two $2$-categories.
	A \ndef{normalised oplax $2$-functor}\index{normalised oplax $2$-functor}
	$F \colon A \to B$ is given by:
	\begin{itemize}
		\item[-] a map $\Ob(A) \to \Ob(B)$ that
		to any object $x$ of $A$ associates an object
		$F(a)$ of $B$;
		\item[-] a map $\Fl_1(A) \to \Fl_1(B)$ that
		to any $1$-cell $f \colon x \to y$ of $A$ associates
		a $1$-cell $F(f) \colon F(x) \to F(y)$ of $B$;
		\item[-] a map $\Fl_2(A) \to \Fl_2(B)$ that
		to any $2$-cell $\alpha \colon f \to g$ of $A$
		associates a $2$-cell $F(\alpha) \colon F(f) \to F(g)$
		of $B$;
		\item[-] a map that to any composable $1$-cells
		$x \xrightarrow{f} y \xrightarrow{g} z$
		of $A$ associates a $2$-cell
		\[
		F(g, f) \colon F(g\ast_0 f) \to F(g) \ast_0 F(f)
		\]
		of $B$.
	\end{itemize}
	These data are subject to the following coherences:
	\begin{description}
		\item[normalisation] for any object $x$ of $A$ (resp.~any $1$-cell $f$ of $A$)
		we have $F(1_x) = 1_{F(x)}$ (resp.~$F(1_f) = 1_{F(f)}$); moreover for
		any $1$-cell $f\colon x \to y$ of $A$ we have
		\[F(1_y, f) = 1_{F(f)} = F(f, 1_x)\,;\]
		\item[cocycle] for any triple
		$x \xto{f} y \xto{g} z \xto{h} t$
		of composable $1$-cells of $A$ we have
		\[
		\bigl(F(h) \ast_0 F(g, f)\bigr) \ast_1 F(h, g\ast_0 f)
		= \bigl(F(h, g) \ast_1 F(f)\bigr) \ast_1 F(h\ast_0 g, f)\,;
		\]
		
		\item[vertical compatibility] for any pair
		\[
		\begin{tikzcd}[column sep=4.5em]
		a\phantom{'}
		\ar[r, bend left=50, looseness=1.2, "f", ""{below, name=f}]
		\ar[r, "g" description, ""{name=gu}, ""{below, name=gd}]
		\ar[r, bend right=50, looseness=1.2, "h"', ""{name=h}]
		\ar[Rightarrow, from=f, to=gu, "\alpha"]
		\ar[Rightarrow, from=gd, to=h, "\beta"]&
		a'
		\end{tikzcd}
		\]
		of $1$-composable $2$-cells $\alpha$ and $\beta$ of $A$, we have
		$F(\beta\ast_1 \alpha) = F(\beta) \ast_1 F(\alpha)$;
		
		\item[horizontal compatibility] for any pair
		\[
		\begin{tikzcd}[column sep=4.5em]
		\bullet
		\ar[r, bend left, "f", ""{below, name=f1}]
		\ar[r, bend right, "f'"', ""{name=f2}]
		\ar[Rightarrow, from=f1, to=f2, "\alpha"]
		&
		\bullet
		\ar[r, bend left, "g", ""{below, name=g1}]
		\ar[r, bend right, "g'"', ""{name=g2}]
		\ar[Rightarrow, from=g1, to=g2, "\beta"]
		& \bullet
		\end{tikzcd}
		\]
		of $0$-composable $2$-cells $\alpha$ and $\beta$ of $A$,
		we have
		\[
		F(g', f') \ast_1 F(\beta \ast_0 \alpha) =
		\bigl( F(\beta) \ast_0 F(\alpha)\bigr) \ast_1 F(g, f)\,.
		\]
	\end{description}

\begin{rem}\label{rem:app-triangles}
	With the notations of the previous paragraph,
	consider a diagram
	\[
		\begin{tikzcd}[column sep=small]
			x \ar[dr, "g"'] \ar[rr, "f", ""{swap, name=f}] && y \\
			 & y \ar[ur, equal] &
			 \ar[from=f, to=2-2, Rightarrow, shorten >= 3pt, "\alpha"{swap, pos=0.4}, "\simeq"{pos=0.4}] 
		\end{tikzcd}
	\]
	in \(A\) where the \(2\)-cell \(\alpha\) is invertible,
	that is there is a \(2\)-cell \(\beta \colon g \to f\)
	such that \(\beta \ast_1 \alpha = 1_f\) and
	\(\alpha \ast_1 \beta = 1_g\). The conditions of normalisation
	and the vertical compatibility then imply that the normalised oplax \(2\)-functor \(F\) maps the above diagram to a
	diagram
	\[
	\begin{tikzcd}[column sep=small]
		Fx \ar[dr, "Fg"'] \ar[rr, "Ff", ""{swap, name=f}] && Fy \\
		& y \ar[ur, equal] &
		\ar[from=f, to=2-2, Rightarrow, shorten >= 3pt, "F\alpha"{swap, pos=0.4}] 
	\end{tikzcd}
	\]
	in \(B\), where \(F\alpha\) is invertible.
	Said otherwise, the normalised oplax functor \(F\)
	maps a \(2\)-simplex of the Duskin nerve of \(A\)
	of the form above, to a \(2\)-simplex of the Duskin nerve of \(B\) with the same properties. 
\end{rem}

\begin{rem}
Normalised oplax functors of 2-categories are not invariant under biequivalence of \(2\)-categories.
For example, suppose that \(\C\) is a 2-category whose mapping categories are all singletons (2-categories with this property are called \emph{codiscrete}), and \(\D\) is a 2-category with a single object \(\ast\) whose endomorphism category is a monoidal category 
\(\V := \Map_{\D}(\ast,\ast)\). Then it is a well-known 
that normalised oplax functors
\(\C \rightsquigarrow\D\) correspond to \(\V^\op\)-enriched categories having \(\Ob(\C)\) as class of objects, and such that the identity maps \(\mathbf{I}_{\V} \to \Map(x,x)\) are all isomorphisms, where \(\mathbf{I}_{\V}\) denotes the monoidal unit of \(\V\). Let us refer to such enriched categories as \ndef{normalized} categories.
Since the 2-category \(\C\) is codiscrete it is biequivalent to a point. Nonetheless, every normalized \(\V^{\op}\)-enriched category with one object is necessarily trivial (its unique mapping object is \(\mathbf{I}_{\V}\)) while this is not true in general if \(\C\) has more than one object (take for example \(\C\) to be the codiscete 2-category on two objects, \(\V=\Set^\op\), and the normalized category corresponding to the discrete category on two objects).
This phenomenon can be considered as a consequence of the fact that, in general, normalised oplax 2-functors fail to preserve
invertible 1-morphisms. 
\end{rem}

	Given two normalised oplax $2$-functors $F \colon A \to B$
	and $G \colon B \to C$, there is an obvious candidate for
	the composition $GF \colon A \to C$ and one checks that
	this is still a normalised oplax $2$-functor; furthermore,
	the identity functor on a category is clearly an identity
	element for normalised oplax $2$-functor too. Hence,
	there is a category $\widetilde{\nCat{2}}$ with small
	$2$-categories as objects and normalised oplax $2$-functors
	as morphisms.
	
	There is also a standard cosimplicial object $\Delta \to \widetilde{\nCat{2}}$
	inducing a nerve functor
	$\widetilde{N_2} \colon \widetilde{\nCat{2}} \to \s$.
	For any $n \ge 0$, the normalised oplax $2$-functors
	$[n] \to A$ correspond precisely to $2$\nbd-func\-tors
	$\On{n} \to A$
	(see, for instance, \cite[\href{https://kerodon.net/tag/00BE}{Tag 00BE}]{LurieKerodon}), where by \(\On{n}\) we denote
	the \(2\)-truncated Street \(n\)-th oriental (see~\cite{StreetOrientals}).
	
	Hence, we get a triangle diagram of functors
	\[
	\begin{tikzcd}[column sep=small]
	& \s & \\
	\nCat{2} \ar[ru, "N_2"] \ar[rr, hook] &&
	\widetilde{\nCat{2}} \ar[lu, "\widetilde{N_2}"']
	\end{tikzcd} ,
	\]
	where \(N_2\) is the Duskin nerve, which is commutative (up to a canonical isomorphism).

\begin{rem}\label{rem:Duskin_ff}
	It is a standard fact that the functor $\widetilde{N_2}\colon \widetilde{\nCat{2}} \to \s$
	is fully faithful (see, for instance~\cite{BullejosFaroBlancoNerve}, \cite{LackPaoliNerves}
	or~\cite[\href{https://kerodon.net/tag/00AU}{Tag 00AU}]{LurieKerodon}).
	With this at hand, one sees that the preservation of triangles
	of the kind described in Remark~\ref{rem:app-triangles}
	can be realized as an immediate consequence of the simplicial
	interpretation of normalised oplax \(2\)-functors.
\end{rem}

\subsection{Oplax functors of \(\infty\)-bicategories}

The notion of a normalize oplax functor was generalized to the \((\infty,2)\)-categorical setting by Gaitsgory--Rozenblyum in~\cite[Chapter 10]{GaitsgoryRozenblyumStudy}. For this, the authors of loc.\ cit.\ use complete Segal \(\infty\)-categories to model \((\infty,2)\)-categories. More precisely, an \((\infty,2)\)-category \(\C\) is encoded via a cocartesian fibration \(\C^{\oint} \to \Del^\op\) whose classifing functor \(\Del^\op \to \Cat_{\infty}\) is a complete Segal object in~\(\Cat_{\infty}\) in the sense of Definition 1.2.7 and 1.2.10 of~\cite{LurieGoodwillie}. In this model a functor of \((\infty,2)\)-categories is encoded via a map
\begin{equation}\label{e:map-C-int}
\begin{tikzcd}
\C^{\oint} \ar[rr,"\phi"] \ar[dr]&& \D^{\oint} \ar[dl]\\
& \Delta^{\mathrm{op}}
\end{tikzcd}
\end{equation}
over \(\Delta^{\op}\) which preserves cocartesian edges. They then define the notion of an oplax functor of \((\infty,2)\)-categories by weakening the preservation of cocartesian edges condition. More precisely, call a map \(\rho\colon [m] \to [n]\) of \(\Del^{\op}\) \ndef{idle} if the image of \(\rho\) is a segment \(\{i \in [n]| a \leq i \leq b\}\) for some \(a\leq b\) in \([n]\). The authors of~\cite{GaitsgoryRozenblyumStudy} then define the notion of an oplax functor of \((\infty,2)\)-categories to be a map over \(\Del^{\op}\) as in~\eqref{e:map-C-int} which is only assumed to preserve cocartesian edges lying over idle maps.

In the present section we offer a definition of oplax functors in the setting of scaled simplicial sets. We expect the two definitions to be equivalent, proof of which will be the topic of future work. The definition of oplax functors given below will serve us in \S\ref{sec:universal} in order to formulate a universal property of the Gray tensor product. In~\cite{GaitsgoryRozenblyumStudy} the authors use a similar universal property in order to define the Gray product in their setting. In particular, any future comparison between the present notion of lax functors with that of~\cite{GaitsgoryRozenblyumStudy} will automatically yield a comparison of the two notions of Gray products.

\begin{define}\label{d:oplax-functors}
Let \(\C = (\ovl{\C},T_{\C})\) and \(\E = (\ovl{\E},T_{\E})\) be two \(\infty\)-bicategories. We will denote by \(L_{\C} \subseteq T_{\C}\) the collection of those thin triangles \(\sig \in T_{\C}\) such that either \(\sig|_{\Del^{\{0,1\}}}\) or \(\sig|_{\Del^{\{1,2\}}}\) is invertible in \(\C\).  
By an \emph{oplax functor} from \(\C\) to \(\E\) we will mean a map of scaled simplicial sets \(\vphi\colon(\ovl{\C},L_{\C}) \to \E\). The collection of oplax functors can be organized into an \(\infty\)-bicategory \(\Fun_{\oplax}(\C,\E) := \Fun((\ovl{\C},L_{\C}),\E)\) using the internal mapping objects of \(\Ss\), see \S\ref{sec:scaled}. 
\end{define}

We first verify that the above definition is homotopically sound. For this, note first that every map \(\C \to \D\) of \(\infty\)-bicategories sends invertible edges to invertible edges and hence maps \(L_{\C}\) into \(L_{\D}\). We then have the following:

\begin{lemma}\label{l:sound}
If \(\vphi\colon \C \to \D\) is an bicategorical equivalence of \(\infty\)-bicategories then the induced map \((\ovl{\C},L_{\C}) \to (\ovl{\D},L_{\D})\) is a bicategorical equivalence. In particular, in this case for every \(\infty\)-bicategory the restriction functor
\[ \Fun_{\oplax}(\D,\E) \to \Fun_{\oplax}(\C,\E) \]
is an equivalence of \(\infty\)-bicategories.
\end{lemma}
\begin{proof}
By Lemma~\ref{l:test-equiv} any bicategorical equivalence of \(\infty\)-bicategories admits a homotopy inverse \(\vphi'\colon \D \to \C\) such that \(\vphi \circ \vphi'\) and \(\vphi' \circ \vphi\) are equivalent to the identities in the \(\infty\)-groupoids \(\cFun(\C,\C)\) and \(\cFun(\D,\D)\) respectively. These equivalences are encoded by maps of scaled simplicial sets \(\eta\colon \Del^1_{\flat} \times \C \to \C\) and \(\eta'\colon \Del^1_{\flat} \times \D \to \D\), in which \(\eta(\Del^1_{\flat}\times \{c\})\) (resp. \(\eta'(\Del^1_{\flat}\times \{d\})\)) is invertible in \(\C\) (resp. \(\D\)) for every vertex \(c\) in \(\C\) (resp. every vertex \(d\) in \(\D\)).

As mentioned above, since \(\vphi\) and \(\vphi'\) preserve invertible edges they preserve the oplax scaling \(L_{\C}\) and \(L_{\D}\) on both sides. We now claim that the homotopies \(\eta,\eta'\) also preserve the oplax scaling in the sense that they extend to maps of scaled simplicial sets
\begin{equation}\label{eta-L} 
\Del^1_{\flat} \times (\ovl{\C},L_{\C}) \to (\ovl{\C},L_{\C}) \quad\text{and}\quad \Del^1_{\flat} \times (\ovl{\D},L_{\D}) \to (\ovl{\D},L_{\D}).
\end{equation}
We prove the claim for the map on the left, the argument for that on the right proceeds in a similar manner. Observe that by the definition of \(L_{\C}\) it will suffice to show that \(\eta\) sends every arrow in \(\Del^1 \times \C\) whose \(\C\)-component is invertible to an invertible arrow in \(\C\). Indeed, let \(f\colon x \to y\) be an invertible arrow encoded by a map \(e\colon \Del^1 \to \C\). Consider the composite
\[\sig=(\Id,e)\colon \Del^1_{\flat} \times \Del^1_{\flat} \to \Del^1_{\flat} \times \C \to \C.\]
We note that since all triangles in \(\Del^1_{\flat}\) are degenerate it follows that all triangles in \(\Del^1_{\flat} \times \Del^1_{\flat}\) are thin. In particular \(\sig\) determines a commutative square in \(\C^{\thi}\) of the form
\[ \xymatrix{
x \ar[r]^{f}\ar[d] & y \ar[d] \\
x' \ar[r] & y'\\
}\]
in which the top horizontal arrow and both vertical arrows are invertible. By the 2-out-of-3 property for invertible arrows we deduce that all arrows in this square (including the diagonal arrow \(x \to y'\)) are invertible in \(\C\). We may then conclude that \(\eta\) sends every arrow in \(\Del^1_{\flat} \times \C\) whose \(\C\)-component is invertible to an invertible edge in \(\C\), and in particular restrict to maps of scaled simplicial sets~\eqref{eta-L}. 

To finish the proof, we now note that every invertible arrow in \(\C\) is also invertible when considered in the non-fibrant scaled simplicial set \((\ovl{\C},L_{\C})\) (in the sense of Definition~\ref{d:equivalence}), since the triangles exhibiting their inverses are included in \(L_{\C}\) by definition. In particular, these edges are sent to invertible edges by any map of scaled simplicial sets \((\ovl{\C},L_{\C}) \to \E\). We may thus conclude that for every \(\infty\)-bicategory \(\E\) the inverse functor \(\vphi'\colon \D \to \C\) and the homotopies \(\eta\) and \(\eta'\) determine a homotopy inverse for the restriction functor
\[ \Fun((\ovl{\D},L_{\D}),\E) \to \Fun((\ovl{\C},L_{\C}),\E)\]
which is consequently an equivalence of \(\infty\)-bicategories.

\end{proof}

\begin{rem}
If we restrict to scaled simplicial sets which are the Duskin nerves of 2-categories 
then we recover those normalised oplax 2-functors that preserve equivalences (rather than just identities), which is what one might expect in light of the homotopy soundness established in Lemma~\ref{l:sound}. On the other hand, the fully-faithfulness of the Duskin nerve (see Remark~\ref{rem:Duskin_ff}) might suggests that, for \(\infty\)-bicategories \(\C = (\ovl{\C},T_{\C})\) and \(\D = (\ovl{\D},T_{\D})\), the direct analogue of the notion of a normalised oplax functor of 2-categories should simply be maps \(\ovl{\C} \to \ovl{\D}\) between the underlying simplicial sets. In fact, these automatically send thin triangles in \(\C\) with one external legs degenerate to thin triangles in \(\D\) (cf.\ Remark~\ref{rem:app-triangles}), but not necessarily all thin triangles with one external leg invertible. However, this notion of an oplax functor \(\C \rightsquigarrow \D\) is not homotopically sound, since the operation \(\C \mapsto \ovl{\C}\) does not send equivalences of \(\infty\)-bicategories to bicategorical equivalences of scaled simplicial sets. For example, if \(J\) is the nerve of the walking isomorphism with two objects then \(J_{\sharp} \to \Del^0\) is an equivalence of \(\infty\)-bicategories but \(J_{\flat} \to \Del^0\) is not a bicategorical equivalence.
\end{rem}

\subsection{The universal property of the Gray product}\label{sec:universal}

In this section we will endow the Gray tensor product of \S\ref{sec:gray-lax} with a \emph{universal mapping property} defined in terms of the \(\infty\)-bicategory of lax functors above. To formulate it, let \(\C = (\ovl{\C},T_{\C})\) and \(\D=(\ovl{\D},T_{\D})\) be two \(\infty\)-bicategories. Let \(T_{\minisimeq} \subseteq T_{\C} \times T_{\D}\) be the subset consisting of those \((\alp,\beta)\) such that either \(\alp|_{\Del^{\{1,2\}}}\) is invertible in \(\C\) or \(\beta|_{\Del^{\{0,1\}}}\) is invertible in \(\D\). By Corollary~\ref{p:invertible-leg} the inclusion
\[ \C \otimes \D \hrar (\ovl{\C} \times \ovl{\D},T_{\minisimeq}) \]
is a bicategorical equivalence. Since this map is also an isomorphism on the level of the underlying simplicial sets we get that for every \(\infty\)-bicategory \(\E\) the associated restriction map gives an isomorphism of (fibrant) scaled simplicial sets
\[ \Fun((\ovl{\C} \times \ovl{\D},T_{\minisimeq}),\E) \cong \Fun(\C \otimes \D,\E).\]
On the other hand, the collection of triangles \(T_{\simeq}\) also contains the set of triangles \(L_{\C \times \D} \subseteq T_{\C} \times T_{\D}\) of Definition~\ref{d:oplax-functors}. Given an \(\infty\)-bicategory \(\E\) we then get a restriction 
\begin{equation}\label{e:restriction-gray-oplax}
\Fun(\C \otimes \D,\E) \cong \Fun((\ovl{\C} \times \ovl{\D},T_{\minisimeq}),\E) \to \Fun_{\oplax}(\C \times \D,\E).
\end{equation}
The universal mapping property of the Gray product can now be formulated as follows:

\begin{thm}
For an \(\infty\)-bicategory \(\E\) the restriction functor~\eqref{e:restriction-gray-oplax} is fully-faithful and its essential image consists of those oplax functors \(\vphi\colon\C \times \D \rightsquigarrow \E\) which satisfy the following conditions:
\begin{enumerate}
\item[(i)]
for all objects \(x \in \C, y \in \D\), the restrictions of \(\vphi\) to \(\{x\} \times \D\) and \(\C \times \{y\}\) are maps of scaled simplicial sets.
\item[(ii)] for all arrows \(f\colon x \to x'\) in \(\C\) and \(g\colon y \to y'\) in \(\D\), the 2-simplex in \(\C\times \D\) depicted by 
\[\begin{tikzcd}[column sep=4em, row sep=3em]
\bigl(x,y\bigr) \ar[r,"(f{,}y)"] \ar[dr,"(f{,}g)"{swap}]& \bigl(x',y\bigr) \ar[d,"(x'{,}g)"]\\
& \bigl(x',y'\bigr)
\end{tikzcd}\]
is mapped by \(f\) to a thin triangle in \(\E\).
\end{enumerate}
\end{thm}
\begin{proof}

Since the map \((\ovl{\C} \times \ovl{\D},L_{\C \times \D}) \to ((\ovl{\C} \times \ovl{\D},T_{\simeq})\) is an isomorphism on the underlying simplicial sets it follows that the restriction functor~\eqref{e:restriction-gray-oplax} is an inclusion of simplicial sets whose image is completely determined by the image on the level of vertices, and so as a functor between \(\infty\)-bicategories it is indeed fully-faithful. Let 
\[L_{\C \times \D} \subseteq G_{\C \times \D} \subseteq T_{\C} \times T_{\D}\] 
be the intermediate set of triangles consisting of \(L_{\C \times \D}\) as well as all those 2-simplices \((\alp,\beta)\) such that either \(\alp\) degenerates to a point, or \(\beta\) degenerates to a point, or \(\alp\) degenerates along \(\Del^{\{1,2\}}\) and \(\beta\) degenerates along \(\Del^{\{0,1\}}\). Unwinding the definitions we see that a map \((\ovl{\C} \times \ovl{\D},L_{\C \times \D}) \to \E\) satisfies conditions (i) and (ii) above if and only if it sends all the triangles in \(G_{\C \times \D}\) to thin triangles in \(\E\).
To finish the proof it will hence suffice to show that  
every thin 2-simplex \(T_{\simeq}\) is in the saturated closure of \(G_{\C \times \D}\). 

\begin{figure}[H]
\begin{tikzpicture}
		\draw [->]  (-3,3) -- node [above, near end] {$S$} (-0.3,3);
		\draw [->]  (-2.8,3.2) -- node [left, near end] {$T$} (-2.8,0.5);
		\foreach \i in {1, 3, 5} {
			\foreach \j in {1, 3, 5} {
				\draw [fill] (\i,\j) circle [radius=0.05];
				
			}
		}
		\draw [->, cyan] (1.1,4.9) -- (2.9,3.1);
		\draw [->, cyan] (3.1,2.9) -- (4.9,1.1);
		\draw [->, magenta] (1.1,5) -- (2.9,5);
		\draw [->, magenta] (3,2.9) -- (3,1.1);
		\draw [->, magenta] (3,4.9) -- (3,3.1);
		\draw [->, magenta] (3.1,1) -- (4.9,1);
		\node at (1,4.6) [magenta] {$0$};
		\node at (2.7,4.6) [magenta] {$1$};
		\node at (2.7,2.7) [magenta] {$2$};
		\node at (2.7,0.6) [magenta] {$3$};
		\node at (5.3,0.6) [magenta] {$4$};
	\end{tikzpicture}
	\caption{The 4-simplex \(\Beta\) (magenta) and the diagonal (cyan)}
\label{Beta}
\end{figure}
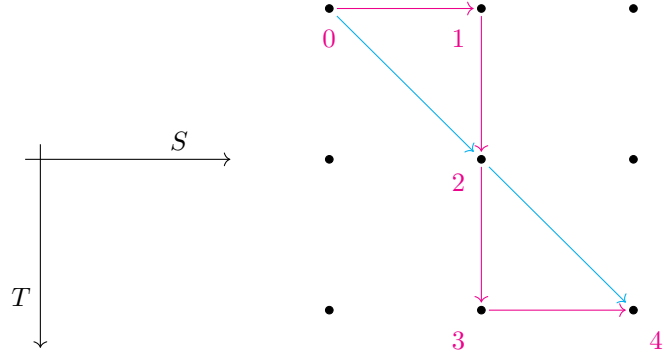

Pick a 2-simplex \((\alp,\beta) \in T_{\simeq}\). 
Then either \(\alp|_{\Del^{\{1,2\}}}\) is invertible in \(\C\) or \(\beta_{\Del^{\{0,1\}}}\) is invertible in \(\D\). To fix ideas assume we are in the former case (the argument in the latter case is entirely similar).
Consider the 4-simplex \(B\colon \Delta^4\to \Delta^2\times \Delta^2\) spanned by the vertices \(((0,0),(1,0),(1,1),(1,2),(2,2))\), and let \(\Beta\x{\text{def}}{=}(\alpha\times\beta)\circ B\colon \Delta^4 \to \C\times \D\) as depicted in Figure \ref{Beta}.

Here, \(\Beta|_{\Del^{\{0,1,2\}}}\), \(\Beta|_{\Del^{\{0,1,3\}}}\) and \(\Beta_{\Del^{\{1,2,3\}}}\) belong to \(G_{\C \times \D}\).
Using \(\Beta|_{\Del^{\{0,1,2,3\}}}\) we then get that \(\Beta|_{\Del^{\{0,2,3\}}}\) is in the saturated closure of \(G_{\C \times \D}\). 
Considering now the face \(\Beta|_{\Del^{\{0,2,3,4\}}}\), we see that since \(\Beta|_{\Del^{\{0,3,4\}}}\) and \(\Beta|_{\Del^{\{1,3,4\}}}\) belong to \(L_{\C \times \D} \subseteq G_{\C \times \D}\) it follows that \(\Beta|_{\Del^{\{0,2,4\}}}\) is also in the saturated closure of \(G_{\C \times \D}\), as desired.
\end{proof}

%
%
\bibliographystyle{amsplain}
\bibliography{biblio}
\end{document}